\documentclass[a4paper,10pt]{amsart}
\usepackage{amsmath, amssymb, hyperref, color}

\theoremstyle{plain}
\newtheorem{theorem}{\bf Theorem}[section]
\newtheorem{proposition}[theorem]{\bf Proposition}
\newtheorem{lemma}[theorem]{\bf Lemma}
\newtheorem{corollary}[theorem]{\bf Corollary}

\theoremstyle{definition}
\newtheorem{example}[theorem]{\bf Example}

\newtheorem{remark}[theorem]{\bf Remark}

\newcommand{\N}{\mathbb N}
\newcommand{\Z}{\mathbb Z}
\newcommand{\R}{\mathbb R}
\newcommand{\Q}{\mathbb Q}

\newcommand{\LK}{\,[\![}
\newcommand{\RK}{]\!]}

\DeclareMathOperator{\spec}{spec}\DeclareMathOperator{\supp}{supp}
\DeclareMathOperator{\Pic}{Pic}

\DeclareMathOperator{\Int}{Int}

\renewcommand{\time}{\negthinspace\times\negthinspace}
\newcommand{\red}{{\text{\rm red}}}
\newcommand{\mon}{\text{\rm mon}}

\newcommand{\BF}{\text{\rm BF}}
\newcommand{\FF}{\text{\rm FF}}

\newcommand{\DP}{\negthinspace :\negthinspace}

\numberwithin{equation}{section}

\begin{document}
\title{On the arithmetic of stable domains}

\address{Institut f\"ur Mathematik und wissenschaftliches Rechnen, Karl-Franzens-Universit\"at Graz, NAWI Graz, Heinrichstra{\ss}e 36, 8010 Graz, Austria}
\email{aqsa.bashir@uni-graz.at,alfred.geroldinger@uni-graz.at,andreas.reinhart@uni-graz.at}
\urladdr{https://imsc.uni-graz.at/geroldinger, https://imsc.uni-graz.at/reinhart/}
\author{Aqsa Bashir and Alfred Geroldinger and Andreas Reinhart}

\thanks{This work was supported by the Austrian Science Fund FWF, Project Numbers J4023-N35, W1230, and P33499.}

\begin{abstract}
A commutative ring $R$ is stable if every non-zero ideal $I$ of $R$ is projective over its ring of endomorphisms. Motivated by a paper of Bass in the 1960s, stable rings have received wide attention in the literature ever since then. Much is known on the algebraic structure of stable rings and on the relationship of stability with other algebraic properties such as divisoriality and the $2$-generator property. In the present paper we study the arithmetic of stable integral domains, with a focus on arithmetic properties of semigroups of ideals of stable orders in Dedekind domains.
\end{abstract}

\maketitle

\section{Introduction}\label{1}

Motivated by a paper of Bass (\cite{Ba63a}), Lipman, Sally and Vasconcelos (\cite{Li71a, Sa-Va73}) introduced the concept of stable ideals and stable rings, which has received wide attention in the literature ever since then. In the present paper we restrict to integral domains. Let $R$ be a commutative integral domain. A non-zero ideal $I\subset R$ is stable if it is projective over its ring of endomorphisms (equivalently, if it is invertible as an ideal of the overring $(I\DP I)$ of $R$). The domain $R$ is called (finitely) stable if every non-zero (finitely generated) ideal of $R$ is stable. By definition, invertible ideals are stable and this implies that Dedekind domains are stable and Pr\"ufer domains are finitely stable. On the other hand, stable domains need neither be noetherian, nor one-dimensional, nor integrally closed.
For background on stable rings, their applications, and for results till 2000 we refer to the survey \cite{Ol01a} by Olberding. Since then stable rings and domains were studied in a series of papers by Bazzoni, Gabelli, Olberding, Roitman, Salce, and others (e.g., \cite{Ga14a, Ga-Ro19a, Ga-Ro19b, Ol01b, Ol02a,Ol14a, Ol14b, Ol16a}).

The goal of the present paper is to study the arithmetic of stable domains, by building on the existing algebraic results. Mori domains and Mori monoids play a central role in factorization theory of integral domains. Every Mori domain $R$ is a BF-domain (this means that every non-zero non-unit element $a\in R$ has a factorization into irreducible elements and the set $\mathsf L(a)\subset\N$ of all possible factorization lengths is finite). For every Mori domain $R$, the monoid $\mathcal I_v^*(R)$ of $v$-invertible $v$-ideals is a Mori monoid. Our starting point is a recent result by Gabelli and Roitman (\cite{Ga-Ro19a}) stating that a domain is stable and Mori if and only if it is one-dimensional stable (Proposition~\ref{3.1}). This implies that stable Mori domains with non-zero conductor to their complete integral closure are stable orders in Dedekind domains (Theorem~\ref{3.7}), and these domains are in the center of our interest.

In Section~\ref{2} we put together some basics on monoids and domains. In Section~\ref{3} we first gather structural results on stable domains (Propositions~\ref{3.1} to~\ref{3.5}). Then we apply them to domains which are of central interest in factorization theory, namely seminormal domains, weakly Krull domains, and Mori domains. In Section~\ref{4}, we study semigroups of $r$-ideals in the setting of ideal systems of cancellative monoids. We derive structural algebraic results and use them to understand when such semigroups of $r$-ideals are half-factorial. Section~\ref{5} contains our main arithmetical results. The main purpose of Section~\ref{5} is to highlight the arithmetical advantages of stability in the context of orders in Dedekind domains. In particular, we show that a series of properties, valid in orders in quadratic number fields (which are stable), also hold true for general stable orders in Dedekind domains. The main result of Section~\ref{5} is Theorem~\ref{5.10}. Among others, it states that  the monoid of non-zero ideals and the monoid of invertible ideals of a stable order in a Dedekind domain are transfer Krull if and only if they are half-factorial (this means that they are transfer Krull only in the trivial case). This is in contrast to a recent result on Bass rings (which are stable). In \cite[Theorem 1.1]{Ba-Sm21}, Baeth and Smertnig show that the monoid of isomorphism classes of finitely generated torsion-free modules over a Bass ring is always transfer Krull.

\section{Background on monoids and domains}\label{2}
\smallskip

We denote by $\N$ the set of positive integers and by $\N_0$ we denote the set of non-negative integers. For rational numbers $a, b\in\Q$, $[a,b]=\{x\in\Z\mid a\le x\le b\}$ is the discrete interval between $a$ and $b$. For subsets $A, B\subset\Z$, $A+B=\{a+b\mid a\in A, b\in B\}$ denotes their sumset. The {\it set of distances} $\Delta(A)$ is the set of all $d\in\N$ for which there is $a\in A$ such that $A\cap [a, a+d]=\{a, a+d\}$. If $A\subset\N$, then $\rho(A)=\sup(A)/\min(A)\in\Q_{\ge 1}\cup\{\infty\}$ is the {\it elasticity} of $A$, and we set $\rho(\{0\})=1$.

Let $H$ be a multiplicatively written commutative semigroup with identity element. We denote by $H^{\times}$ the group of invertible elements of $H$. We say that $H$ is {\it reduced} if $H^{\times}=\{1\}$ and we denote by $H_{\red}=\{aH^{\times}\mid a\in H\}$ the {\it associated reduced semigroup} of $H$. An element $u\in H$ is said to be {\it cancellative} if $au=bu$ implies $a=b$ for all $a,b\in H$. The semigroup $H$ is said to be
\begin{itemize}
\item {\it cancellative} if all elements of $H$ are cancellative;
\item {\it unit-cancellative} if $a,u\in H$ and $a=au$ implies that $u\in H^{\times}$.
\end{itemize}
Clearly, every cancellative monoid is unit-cancellative. We will study semigroups of ideals that are unit-cancellative but not necessarily cancellative.

\smallskip
\centerline{\it Throughout, a monoid means a}
\centerline{\it commutative unit-cancellative semigroup with identity element.}

\smallskip
For a set $P$, we denote by $\mathcal F(P)$ the free abelian monoid with basis $P$. Elements $a\in\mathcal F(P)$ are written in the form
\[
a=\prod_{p\in P} p^{\mathsf v_p(a)}\,,\quad\text{where $\mathsf v_p\colon\mathcal F(P)\to\N_0$ }
\]
is the $p$-adic valuation. We denote by $|a|=\sum_{p\in P}\mathsf v_p(a)\in\N_0$ the {\it length} of $a$ and by $\supp(a)=\{ p\in P\mid\mathsf v_p(a) > 0\}\subset P$ the {\it support} of $a$.

Let $H$ be a monoid. A non-unit $a\in H$ is said to be an {\it atom} (or {\it irreducible}) if $a=bc$ with $b,c\in H$ implies that $b\in H^{\times}$ or $c\in H^{\times}$. We denote by $\mathcal A(H)$ the {\it set of atoms} of $H$ and we say that $H$ is {\it atomic} if every non-unit is a finite product of atoms. Two elements $a,b\in H$ are called {\it associated} if $a=bc$ for some $c\in H^{\times}$. If $a=u_1\cdot\ldots\cdot u_k\in H$, where $k\in\N$ and $u_1,\ldots,u_k\in\mathcal A(H)$, then $k$ is a {\it factorization length} and the set $\mathsf L(a)\subset\N$ of all factorization lengths of $a$ is called the {\it set of lengths} of $a$. For convenience, we set $\mathsf L(a)=\{0\}$ for $a\in H^{\times}$. Then $\mathcal L(H)=\{\mathsf L(a)\mid a\in H\}$ is the {\it system of sets of lengths} of $H$,
\[
\begin{aligned}
\Delta(H) &=\bigcup_{L\in\mathcal L(H)}\Delta(L)\subset\N\quad\text{is the {\it set of distances} of $H$, and}\\
\rho(H) &=\sup\{\rho(L)\mid L\in\mathcal L(H)\}\in\R_{\ge 1}\cup\{\infty\}\quad\text{is the {\it elasticity} of $H$}\,.
\end{aligned}
\]
We say that the {\it elasticity is accepted} if $\rho(H)=\rho(L)$ for some $L\in\mathcal L(H)$. The monoid $H$ is
\begin{itemize}
\item {\it half-factorial} if it is atomic and $|L|=1$ for all $L\in\mathcal L(H)$,
\item an {\it\FF-monoid} if it is atomic and each element of $H$ is divisible by only finitely many non-associated atoms, and
\item a {\it\BF-monoid} if it is atomic and all sets of lengths are finite.
\end{itemize}
By definition, an atomic monoid is half-factorial if and only if $\Delta(H)=\emptyset$ if and only if $\rho(H)=1$.
FF-monoids are BF-monoids, BF-monoids satisfy the ACCP (ascending chain condition on principal ideals), and monoids satisfying the ACCP are atomic and archimedean (i.e., $\bigcap_{n\ge 1} a^nH=\emptyset$ for all $a\in H\setminus H^{\times}$). If $H$ is atomic but not half-factorial, then we have $\gcd\Delta(H)=\min\Delta(H)$.

Suppose that $H$ is cancellative, $\mathfrak m=H\setminus H^{\times}$, and let $\mathsf q(H)$ be the quotient group of $H$. We denote by
\begin{itemize}
\item $H'=\{x\in\mathsf q(H)\mid\text{there is some $N\in\N$ such that $x^n\in H$ for all $n\ge N$}\}$ the {\it seminormal closure} of $H$, and by
\item $\widehat H=\{x\in\mathsf q(H)\mid\text{there is $c\in H$ such that $cx^n\in H$ for all $n\in\N$}\}$ the {\it complete integral closure} of $H$.
\end{itemize}
Then $H\subset H'\subset\widehat H\subset\mathsf q(H)$, and we say that $H$ is {\it seminormal} (resp. {\it completely integrally closed}) if $H=H'$ (resp. $H=\widehat H$). Let $A, B\subset\mathsf q(H)$ be subsets. We set $(A\DP B)=\{z\in\mathsf q(H)\mid zB\subset A\}$ and $A^{-1}=(H\DP A)$. If $A\subset H$, then $A$ is a {\it divisorial ideal} (or a {\it $v$-ideal}) if $A=A_v:=(A^{-1})^{-1}$, and $A$ is an {\it $s$-ideal} if $A=AH$. If $\mathfrak p\subsetneq H$ is an $s$-ideal of $H$, then $\mathfrak p$ is called a {\it prime $s$-ideal} of $H$ if for all $x,y\in H$ with $xy\in\mathfrak p$, it follows that $x\in\mathfrak p$ or $y\in\mathfrak p$. For an $s$-ideal $I$ of $H$ let $\sqrt{I}=\{x\in H\mid$ there is $n\in\N$ such that $x^n\in I\}$ denote the {\it radical} of $I$. The monoid $H$ is said to be
\begin{itemize}
\item {\it Mori} if it satisfies the ascending chain condition on divisorial ideals,
\item {\it Krull} if it is a completely integrally closed Mori monoid,
\item a G-{\it monoid} if the intersection of all non-empty prime $s$-ideals is non-empty,
\item {\it primary} if $H\ne H^{\times}$ and for all $a,b\in\mathfrak m$ there is $n\in\N$ such that $b^n\in aH$,
\item {\it strongly primary} if $H\ne H^{\times}$ and for every $a\in\mathfrak m$ there is $n\in\N$ such that $\mathfrak m^n\subset aH$ (we denote by $\mathcal M(a)$ the smallest $n\in\N$ having this property), and
\item {\it finitely primary} (of rank $s$ and exponent $\alpha$) if $H$ is a submonoid of a factorial monoid $F=F^{\times}\times\mathcal F(\{p_1,\ldots, p_s\})$ such that $\mathfrak m\subset p_1\cdot\ldots\cdot p_sF$ and $(p_1\cdot\ldots\cdot p_s)^{\alpha}F\subset H$.
\end{itemize}
Finitely primary monoids and primary Mori monoids are strongly primary. (To see that the last statement is valid let $H$ be a primary Mori monoid, $\mathfrak m=H\setminus H^{\times}$ and $a\in\mathfrak m$. Then $\sqrt{aH}=\mathfrak m=E_v$ for some finite set $E\subset\sqrt{aH}$. We infer that $E^n\subset aH$ for some $n\in\N$. Therefore, $\mathfrak m^n\subset (E^n)_v\subset aH$, and thus $H$ is strongly primary.) Mori monoids and strongly primary monoids are BF-monoids.

By a {\it domain}, we mean a commutative ring with non-zero identity element and without non-zero zero-divisors. Let $R$ be a domain. We denote by $R^{\bullet}=R\setminus\{0\}$ the {\it multiplicative monoid of non-zero elements}, by $R^{\times}$ the {\it group of units}, by $\overline R$ the {\it integral closure} of $R$, by $\widehat R$ the {\it complete integral closure} of $R$, by $\mathfrak X(R)$ the set of {\it non-zero minimal prime ideals} of $R$, and by $\mathsf q(R)$ the {\it quotient field} of $R$. An ideal $I\subset R$ is called {\it $2$-generated} if there are some $a,b\in I$ such that $I=aR+bR$.
We say that $R$ is atomic (a BF-domain, an FF-domain, a Mori domain, a Krull domain, a G-domain, archimedean, (strongly) primary, seminormal, completely integrally closed) if its monoid $R^{\bullet}$ has the respective property. By \cite[Proposition 2.10.7]{Ge-HK06a}, $R$ is primary if and only if $R$ is one-dimensional and local.
The domain $R$
\begin{itemize}
\item has {\it finite character} if every non-zero element is contained in only finitely many maximal ideals.
\item is {\it divisorial} if every non-zero ideal is divisorial.
\item is {\it h-local} if $R$ has finite character and every non-zero prime ideal of $R$ is contained in a unique maximal ideal of $R$.
\end{itemize}
One-dimensional Mori domains have finite character by \cite[Lemma 3.11]{Ga-Ro19a}.

Let $S$ be an integral domain such that $R\subset S$ is a subring. Then $R$ is an {\it order} in $S$ if $\mathsf q(R)=\mathsf q(S)$ and $S$ is a finitely generated $R$-module. Moreover, the following statements are equivalent if $R$ is not a field (\cite[Theorem 2.10.6]{Ge-HK06a}){\rm\,:}
\begin{itemize}
\item $R$ is an order in a Dedekind domain.
\item $R$ is one-dimensional noetherian and the integral closure $\overline R$ of $R$ is a finitely generated $R$-module.
\end{itemize}
The extension $R\subset S$ is {\it quadratic} if $xy\in xR+yR+R$ for all
$x,y\in S$; equivalently, every $R$-module between $R$ and $S$ is a ring. If $R\subset S$ is quadratic, then, for every $x\in S$, we have
$x^2\in xR+R$; that is, every $ x\in S$ is a root of a monic polynomial of degree at most $2$
with coefficients in $R$. Thus every quadratic extension is an integral extension.

\section{Stable domains}\label{3}
\smallskip

In this section we first gather main properties of stable domains (Propositions~\ref{3.1} to~\ref{3.5}). Then we analyze what consequences stability has on some key classes of domains studied in factorization theory, including seminormal domains, weakly Krull domains, G-domains, and Mori domains (Theorems~\ref{3.6} and~\ref{3.7}).

Let $R$ be a domain. A non-zero ideal $I\subset R$ is {\it stable} if it is invertible as an ideal of the overring $(I\DP I)$ of $R$. The domain is called {\it $($finitely$)$ stable} if every non-zero (finitely generated) ideal of $R$ is stable. Since invertible ideals are obviously stable, Dedekind domains are stable and Pr\"ufer domains are finitely stable. Conversely, if $R$ is completely integrally closed and stable, then $R=(I\DP I)$ for every non-zero ideal $I\subset R$, whence every non-zero ideal is invertible in $R$ and $R$ is a Dedekind domain. Recall that $R$ is an {\it almost Dedekind} domain if $R_{\mathfrak m}$ is a Dedekind domain for each $\mathfrak m\in\max(R)$. Every almost Dedekind domain is a completely integrally closed Pr\"ufer domain, and thus it is finitely stable. Nevertheless, $R$ is a Dedekind domain if and only if $R$ is a stable almost Dedekind domain. In particular, every almost Dedekind domain that is not a Dedekind domain is not stable. For an example of an almost Dedekind domain that is not a Dedekind domain we refer to \cite[Example 35, page 290]{Lo06}. We recall that stable domains need neither be noetherian, nor integrally closed, nor one-dimensional (\cite[Sections 3 and 4]{Ol01a}), and we use without further mention that overrings of stable domains are stable (\cite[Theorem 5.1]{Ol02a}).

\begin{proposition}\label{3.1}
Let $R$ be a domain that is not a field. Then the following statements are equivalent.
\begin{enumerate}
\item[(a)] $R$ is a one-dimensional stable domain.
\item[(b)] $R$ is a finitely stable Mori domain.
\item[(c)] $R$ is a stable Mori domain.
\end{enumerate}
\end{proposition}

\begin{proof}
This is due to Gabelli and Roitman. More precisely,
the equivalence of (a) and (b) is proved in \cite[Theorem 4.8]{Ga-Ro19a}. Clearly, (c) implies (b). If (a) and (b) hold, then (c) holds by \cite[Proposition 4.4]{Ga-Ro19a}.
\end{proof}

Examples given by Olberding in \cite{Ol01b,Ol14b} show that one-dimensional stable domains need not be noetherian. The ring $\Int(\Z)$ of integer-valued polynomials is a two-dimensional completely integrally closed Pr\"ufer domain and a BF-domain. $\Int(\Z)$ is finitely stable (as it is Pr\"ufer) but not stable (as it is not Dedekind). Thus in Statement (b), the property ``Mori'' cannot be replaced by ``BF''. In Example~\ref{3.9}.3 we show that ``Mori'' cannot be replaced by ``BF'' in Statement (c) even if $R$ is a Pr\"ufer domain. Next we consider the local case.

\begin{corollary}\label{3.2}
Let $R$ be a local domain that is not a field.
\begin{enumerate}
\item The following statements are equivalent.
\begin{enumerate}
\item[(a)] $R$ is a one-dimensional stable domain.
\item[(b)] $R$ is a primary stable domain.
\item[(c)] $R$ is a stable Mori domain.
\item[(d)] $R$ is a strongly primary stable domain.
\end{enumerate}
If these conditions hold and $(R\DP\overline R)=\{0\}$ $($for examples, see \cite[Theorem 2.13]{Ol16a}$)$, then $\overline R$ is a discrete valuation domain.
\item If $R$ is one-dimensional, then the following statements are equivalent.
\begin{enumerate}
\item[(a)] $R$ is stable.
\item[(b)] $R$ is finitely stable with stable maximal ideal.
\item[(c)] $\overline R$ is a quadratic extension of $R$ and $\overline R$ is a Dedekind domain with at most two maximal ideals.
\end{enumerate}
\item If $R$ is finitely stable with stable maximal ideal $\mathfrak m$, then the following statements are equivalent.
\begin{enumerate}
\item[(a)] $\bigcap_{n\in\N} {\mathfrak m}^n=\{0\}$.
\item[(b)] $R$ is a \BF-domain.
\item[(c)] $R$ satisfies the {\rm ACCP}.
\item[(d)] $R$ is archimedean.
\end{enumerate}
\end{enumerate}
\end{corollary}

\begin{proof}
1. Since $R$ is one-dimensional and local if and only if $R^{\bullet}$ is primary, Conditions (a) and (b) are equivalent. Conditions (a) and (c) are equivalent by Proposition~\ref{3.1}. Obviously, Condition (d) implies Condition (b). Since primary Mori monoids are strongly primary by \cite[Lemma 3.1]{Ge-Ha-Le07}, Conditions (b) and (c) imply Condition (d). If (a) - (d) hold and $(R\DP\overline R)=\{0\}$, then $\overline R$ is a discrete valuation domain by \cite[Corollary 4.17]{Ol02a}.

2. See \cite[Theorem 4.2]{Ol16a}.

3. (a) $\Rightarrow$ (b) This is an immediate consequence of \cite[Theorem 1.3.4]{Ge-HK06a}.

(b) $\Rightarrow$ (c) This follows from \cite[Corollary 1.3.3]{Ge-HK06a}.

(c) $\Rightarrow$ (d) This is clear (e.g., see page 2 of \cite{Ga-Ro19b}).

(d) $\Rightarrow$ (a) This follows from \cite[Proposition 2.12]{Ga-Ro19b}.
\end{proof}

\smallskip
Let $R$ be a domain. By Corollary~\ref{3.2}.1, every strongly primary stable domain is Mori. This is not true for general strongly primary domains (\cite[Section 3]{Ge-Ro20a}) and it is in strong contrast to other classes of strongly primary monoids (\cite[Theorem 3.3]{Ge-Go-Tr20}). By \cite[Example 5.17]{Ga-Ro19b} there exists a stable two-dimensional archimedean local integral domain. We infer by Corollary~\ref{3.2}.3 that such a domain is a BF-domain. In particular, a local stable BF-domain need not satisfy the equivalent conditions of Corollary~\ref{3.2}.1.

\smallskip
Note that if $R$ is a local domain whose ideals are $2$-generated, then $R$ is finitely stable with stable maximal ideal (e.g. see Proposition~\ref{3.5}.4) and the equivalent conditions in Corollary~\ref{3.2}.3 are satisfied (since $R$ is noetherian). Nevertheless, such a domain is (in general) neither half-factorial nor an FF-domain. In what follows, we present suitable counterexamples.

\smallskip
Let $K$ be a quadratic number field with maximal order $\mathcal{O}_K$ and $p$ be a prime number such that $p$ is split (i.e., $p\mathcal{O}_K$ is the product of two distinct prime ideals of $\mathcal{O}_K)$. (For instance, let $K=\mathbb{Q}(\sqrt{-7})$ and $p=2$.) Let $\mathcal{O}$ be the unique order in $K$ with conductor $p\mathcal{O}_K$ and let $\mathfrak p$ be a maximal ideal of $\mathcal{O}$ that contains the conductor. Set $S=\mathcal{O}_{\mathfrak p}$. Then $S$ is a local domain whose ideals are $2$-generated and there are precisely two maximal ideals of $\overline{S}$ that are lying over the maximal ideal of $S$. It follows from \cite[Theorem 3.1.5.2]{Ge-HK06a} that $S$ is not half-factorial. (Note that $S^{\bullet}$ is a finitely primary monoid of rank two, and thus it has infinite elasticity by \cite[Theorem 3.1.5.2]{Ge-HK06a}. Therefore, it cannot be half-factorial.)

\smallskip
Let $T=\mathbb{R}+X\mathbb{C}[\![X]\!]$. Then $T$ is a local domain with maximal ideal $X\mathbb{C}[\![X]\!]$ and every ideal of $T$ is $2$-generated by Corollary~\ref{3.2}.2 and Proposition~\ref{3.5}.4. Observe that $T$ is not an FF-domain, since $aX,a^{-1}X\in\mathcal{A}(T)$ and $X^2=(aX)(a^{-1}X)$ for each $a\in\mathbb{C}\setminus\{0\}$.

\smallskip
We do not know whether a local atomic finitely stable domain with stable maximal ideal satisfies the equivalent conditions in Corollary~\ref{3.2}.3.

\medskip

\begin{proposition}\label{3.3}
Let $R$ be a domain.
\begin{enumerate}
\item $R$ is finitely stable if and only if $R\subset\overline R$ is a quadratic extension, $\overline R$ is Pr\"ufer, and there are at most two maximal ideals of $\overline{R}$ lying over every maximal ideal of $R$.
\item A semilocal Pr\"ufer domain is stable if and only if it is strongly discrete.
\item $R$ is an integrally closed stable domain if and only of $R$ is a strongly discrete Pr\"ufer domain with finite character if and only if $R$ is a generalized Dedekind domain with finite character.
\item An integrally closed one-dimensional domain is stable if and only if it is Dedekind.
\end{enumerate}
\end{proposition}

\begin{proof}
Recall that a Pr\"ufer domain $R$ is strongly discrete provided that no non-zero prime ideal $P$ of $R$ satisfies $P=P^2$.

1. \cite[Corollary 5.11]{Ol14a}.

2. See \cite[Proposition 2.10]{An-Hu-Pa87} and \cite[Proposition 2.5]{Ga14a}.

3. It is an immediate consequence of \cite[Theorem 4.6]{Ol98a} that $R$ is an integrally closed stable domain if and only if $R$ is a strongly discrete Pr\"ufer domain with finite character. Moreover, it follows from \cite[Corollary 2.13]{Ga14a} that $R$ is integrally closed and stable if and only if $R$ is a generalized Dedekind domain with finite character.

4. Since a domain is Dedekind if and only if it is generalized Dedekind of dimension one (\cite[Proposition 2.1]{Ga06a}, this follows from 3.
\end{proof}

Proposition~\ref{3.3}.4 characterizes integrally closed stable domains, that are one-dimensional. However, there are,
for every $n\in\N$, $n$-dimensional local stable valuation domains (\cite[Example 5.11]{Ga-Ro19b}, and recall that valuation domains are integrally closed).

\begin{lemma}\label{3.4}
Let $R$ be a local domain with maximal ideal $\mathfrak m$ such that $R$ is not a field.
\begin{enumerate}
\item If $R$ is noetherian, then $R$ is divisorial if and only if $R$ is one-dimensional and ${\mathfrak m}^{-1}/R$ is a simple $R$-module.
\item If $R$ is seminormal and one-dimensional, then $(R\DP\widehat{R})\supset\mathfrak m$.
\end{enumerate}
\end{lemma}

\begin{proof}
1. This follows from \cite[Theorem A]{Ba00a}.

2. This is an immediate consequence of \cite[Lemma 3.3]{Ge-Ka-Re15a}.
\end{proof}

\begin{proposition}\label{3.5}
Let $R$ be a domain.
\begin{enumerate}
\item $R$ is divisorial if and only if $R$ is h-local and $R_{\mathfrak m}$ is divisorial for every ${\mathfrak m}\in\max(R)$.
\item $R$ is stable if and only if $R$ is of finite character and $R_{\mathfrak m}$ is stable for every ${\mathfrak m}\in\max(R)$.
\item $R$ is a divisorial Mori domain if and only if $R$ is of finite character and $R_{\mathfrak m}$ is a divisorial Mori domain for every ${\mathfrak m}\in\max(R)$.
\item Every ideal of $R$ is $2$-generated if and only if $R$ is a divisorial stable Mori domain. If $R$ is a stable Mori domain with $(R\DP\overline R)\ne\{0\}$, then $R$ is divisorial and every ideal of $R$ is $2$-generated.
\item Every ideal of $R$ is $2$-generated if and only if $R$ is of finite character and for all ${\mathfrak m}\in\max(R)$, every ideal of $R_{\mathfrak m}$ is $2$-generated.
\end{enumerate}
\end{proposition}

\begin{proof}
1. This follows from \cite[Proposition 5.4]{Ba-Sa96a}.

2. This follows from \cite[Theorem 3.3]{Ol02a}.

3. Without restriction assume that $R$ is not a field. First let $R$ be a divisorial Mori domain. It follows by 1. that $R$ is of finite character and $R_{\mathfrak m}$ is divisorial for all ${\mathfrak m}\in\max(R)$. Clearly, $R_{\mathfrak m}$ is a Mori domain for every ${\mathfrak m}\in\max(R)$.

Now let $R$ be of finite character and let $R_{\mathfrak m}$ be a divisorial Mori domain for every ${\mathfrak m}\in\max(R)$. We infer by \cite[Th\'eor\`eme 1]{Qu76a} that $R$ is a Mori domain. If ${\mathfrak m}\in\max(R)$, then $R_{\mathfrak m}$ is clearly noetherian, and hence $R_{\mathfrak m}$ is one-dimensional by Lemma~\ref{3.4}.1. Therefore, $R$ is one-dimensional, and thus $R$ is h-local. Therefore, $R$ is divisorial by 1.

4. We infer by \cite[Theorems 3.1 and 3.12]{Ol01c} that every ideal of $R$ is $2$-generated if and only if $R$ is a noetherian stable divisorial domain. Clearly, $R$ is noetherian and divisorial if and only if $R$ is a divisorial Mori domain, and hence the first statement follows. If $R$ is a stable Mori domain with $(R\DP\overline R)\ne\{0\}$, then $R$ is at most one-dimensional by Proposition~\ref{3.1}, and thus every ideal of $R$ is 2-generated by \cite[Proposition 4.5]{Ol01b}.

5. This is an immediate consequence of 2., 3. and 4.
\end{proof}

By Proposition~\ref{3.5}.4, orders in quadratic number fields are stable because every ideal is $2$-generated (for background on orders in quadratic number fields we refer to \cite{HK13a}). Much research was done to characterize domains, for which all ideals are $2$-generated (\cite[Theorem 7.3]{Ba-Sa96a}, \cite[Theorem 17]{Go98a}, \cite{Ma70a}). We continue with a characterization within the class of seminormal domains.

\begin{theorem}\label{3.6}
Let $R$ be a seminormal domain. Then the following statements are equivalent.
\begin{enumerate}
\item[(a)] Every ideal of $R$ is $2$-generated.
\item[(b)] $R$ is a divisorial Mori domain.
\item[(c)] $R$ is a finitely stable Mori domain.
\end{enumerate}
\end{theorem}

\begin{proof}
Without restriction assume that $R$ is not a field. Note that if $R$ is of finite character, then $R$ is a Mori domain if and only if $R_{\mathfrak m}$ is a Mori domain for every $\mathfrak m\in\max(R)$ (\cite[Th\'eor\`eme 1]{Qu76a}). We obtain by Proposition~\ref{3.5}.3 that $R$ is a divisorial Mori domain if and only if $R$ is of finite character and $R_{\mathfrak m}$ is a divisorial Mori domain for every $\mathfrak m\in\max(R)$. Besides that we infer by Propositions~\ref{3.1} and~\ref{3.5}.2 that $R$ is a finitely stable Mori domain if and only if $R$ is of finite character and $R_{\mathfrak m}$ is a finitely stable Mori domain for every $\mathfrak m\in\max(R)$. By using Proposition~\ref{3.5}.5 and the fact that $R_{\mathfrak m}$ is seminormal for every $\mathfrak m\in\max(R)$, it suffices to prove the equivalence in the local case. Let $R$ be local with maximal ideal $\mathfrak m$.

(a) $\Rightarrow$ (b) This follows from Proposition~\ref{3.5}.4.

(b) $\Rightarrow$ (c) Observe that $R$ is noetherian, and thus $R$ is one-dimensional by Lemma~\ref{3.4}.1. We infer that $\overline{R}$ is a semilocal principal ideal domain, and thus $\overline{R}$ is a finitely stable Mori domain. In particular, we can assume without restriction that $R$ is not integrally closed. Since $R\subsetneq\overline{R}$, it follows that $(R\DP\overline{R})=(R\DP\widehat{R})=\mathfrak m$ by Lemma~\ref{3.4}.2. Since $R$ is not integrally closed, we have that $\mathfrak m$ is not invertible. Therefore, $\mathfrak m\mathfrak m^{-1}\subset\mathfrak m$. Moreover, $\overline{R}\mathfrak m=\overline{R}(R\DP\overline{R})\subset R$, and hence $\overline{R}\subset\mathfrak m^{-1}$. We infer that $\mathfrak m^{-1}\subset (\mathfrak m\DP\mathfrak m)\subset\overline{R}\subset \mathfrak m^{-1}$, and thus $\overline{R}=\mathfrak m^{-1}$.

Consequently, $\overline{R}/R$ is a simple $R$-module by Lemma~\ref{3.4}.1. In particular, $R\subset\overline{R}$ is a quadratic extension. Observe that $\boldsymbol l_R(\overline{R}/\mathfrak m)=\boldsymbol l_R(\overline{R}/R)+\boldsymbol l_R(R/\mathfrak m)=2$. Set $k=|\{\mathfrak q\in\max(\overline{R})\mid \mathfrak q\cap R=\mathfrak m\}|$. Then $k=|\max(\overline{R})|$. Assume that $k\geq 3$. There are some distinct $\mathfrak q_1,\mathfrak q_2,\mathfrak q_3\in\max(\overline{R})$. Note that $\mathfrak m\subset \mathfrak q_1\cap\mathfrak q_2\cap\mathfrak q_3\subsetneq \mathfrak q_1\cap\mathfrak q_2\subsetneq\mathfrak q_1\subsetneq\overline{R}$, and thus $\boldsymbol l_R(\overline{R}/\mathfrak m)\geq 3$, a contradiction. We infer that $k\leq 2$. It follows from Corollary~\ref{3.2}.2 that $R$ is finitely stable.

(c) $\Rightarrow$ (a) Note that $R$ is a one-dimensional stable domain by Proposition~\ref{3.1}. It follows from Lemma~\ref{3.4}.2 that $\{0\}\not=\mathfrak m\subset (R\DP\widehat{R})\subset (R\DP\overline{R})$, and thus every ideal of $R$ is $2$-generated by Proposition~\ref{3.5}.4.
\end{proof}

A domain $R$ is said to be {\it weakly Krull} if
\[
R=\bigcap_{\mathfrak p\in\mathfrak X(R)} R_{\mathfrak p}\quad\text{and the intersection is of finite character,}
\]
which means that $\{\mathfrak p\in\mathfrak X(R)\mid x\not\in R_{\mathfrak p}^{\times}\}$ is finite for all $x\in R^{\bullet}$. Weakly Krull domains were introduced by Anderson, Anderson, Mott, and Zafrullah (\cite{An-An-Za92b, An-Mo-Za92}), and their multiplicative character was pointed out by Halter-Koch (\cite[Chapter 22]{HK98}).

\smallskip
\begin{theorem}\label{3.7}
Let $R$ be a domain with $(R\DP\widehat{R})\neq\{0\}$, and suppose that $R$ is either weakly Krull or Mori. Then $R$ is stable if and only if every ideal of $R$ is $2$-generated. If this holds, then $R$ is an order in a Dedekind domain.
\end{theorem}

\begin{proof}
If every ideal of $R$ is $2$-generated, then $R$ is stable by Proposition~\ref{3.5}.4. Conversely, let $R$ be stable.

Let us first suppose that $R$ is weakly Krull. Then, for every $\mathfrak p\in\mathfrak X(R)$, $R_{\mathfrak p}$ is one-dimensional and stable, whence Mori by Proposition~\ref{3.1}. Since $R$ is weakly Krull, this implies that $R$ is Mori by \cite[Lemma 5.1]{Ge-Ka-Re15a}.

Thus $R$ is Mori in both cases. Using Proposition~\ref{3.1} again we infer that $R$ is one-dimensional. Therefore, $\overline R$ is one-dimensional integrally closed and stable, whence $\overline R$ is a Dedekind domain by Proposition~\ref{3.3}.4. Since $(R\DP\overline R)\ne\{0\}$, Proposition~\ref{3.5}.4 implies that every ideal of $R$ is $2$-generated.
\end{proof}

\smallskip
\begin{corollary}\label{3.8}
Let $R$ be a seminormal G-domain and suppose that $R$ is either Mori or one-dimensional. Then $R$ is stable if and only if every ideal of $R$ is $2$-generated. If this holds, then $R$ is an order in a Dedekind domain.
\end{corollary}

\begin{proof}
Since $R$ is a seminormal G-domain, $(R\DP\widehat R)\ne\{0\}$ by \cite[Proposition 4.8]{G-HK-H-K03}. Thus the claim follows from Theorem~\ref{3.7}.
\end{proof}

\smallskip
\begin{example}\label{3.9}~

1. There exist integrally closed one-dimensional local Mori domains which are neither valuation domains nor finitely stable. Let $K$ be a field, $Y$ an indeterminate over $K$, and $X$ an indeterminate over $K(Y)$. Then $R=K+XK(Y)\LK X\RK$ is an integrally closed one-dimensional local Mori domain which is not completely integrally closed. Thus, $R$ is not a valuation domain. By Proposition~\ref{3.3}.4, it is not stable because it is not a Dedekind domain, and hence it is not finitely stable by Proposition~\ref{3.1}.

\smallskip
2. There exists a seminormal two-dimensional local stable domain. Let $p\in\mathbb{Z}$ be a prime and $R=\mathbb{Z}_{(p)}+X\mathbb{R}\LK X\RK$. Since $\mathbb{R}\LK X\RK$ is a discrete valuation domain with maximal ideal $\mathfrak m=X\mathbb{R}\LK X\RK$ and also $\mathbb{Z}_{(p)}$ is a discrete valuation domain with maximal ideal $p\mathbb{Z}_{(p)}$; $R$ is a local two-dimensional domain with maximal ideal $\mathfrak n=pR$ (and $\{0\}\subset\mathfrak m\subset\mathfrak n$).
Now $R$ is stable as well by \cite[Theorem 2.6]{Ol01b}. Thus $R$ is not Mori by Proposition~\ref{3.1}.

Moreover, $R$ is seminormal. Indeed, we know $\mathbb{Z}_{(p)}$ is integrally closed in $\mathbb{Q}$ and $R\subset D=\mathbb{Q}+X\mathbb{R}\LK X\RK\subset\mathbb{R}\LK X\RK$. Let $t\in\mathsf q(R)=\mathbb{R}((X))$ with $t^2,t^3\in R$. Then $t^2,t^3\in\mathbb{R}\LK X\RK$, and hence $t\in\mathbb{R}\LK X\RK$ (since $\mathbb{R}\LK X\RK$ is completely integrally closed). We infer that $t_{0}\in\mathbb{R}$ and $t_{0}^2,t_{0}^3\in\mathbb{Q}$. If $t_{0}=0$, then $t\in D$. If $t_{0}\neq 0$, then $t_{0}=t_{0}^3t_{0}^{-2}\in\mathbb{Q}$, that is $t\in D$. In any case $D$ is seminormal. Now clearly $R$ is seminormal in $D$. Therefore, $R$ is seminormal.

3. There exists a two-dimensional stable Pr\"ufer domain $R$ which is a BF-domain, whence $R$ is a finitely stable BF-domain that is not Mori (cf. Proposition~\ref{3.1}). To see this we analyze an example given by Gabelli and Roitman. Let $K$ be a field and let $X$ and $Y$ be independent indeterminates over $K$. Set $S=K[Y]\setminus YK[Y]$ and let $R=S^{-1}(K[\{\frac{X(1-X)^n}{Y^n},\frac{Y^{n+1}}{(1-X)^n}\mid n\in\N_0\}])$. Set $T=\frac{1-X}{Y}$. It is shown in \cite[Example 5.13]{Ga-Ro19b} that $R$ is a two-dimensional stable Pr\"ufer domain which satisfies the ACCP. In particular, $R$ is archimedean. Moreover, it is shown in \cite[Example 5.13]{Ga-Ro19b} that $Y$ and $T$ are algebraically independent over $K$ and $R=S^{-1}(K[\{(1-YT)T^n,\frac{Y}{T^n}\mid n\in\N_0\}])$.

Next we prove that $S^{-1}(K[Y,T,T^{-1}])\subset\widehat{R}$. Observe that $T=\frac{YT}{Y}$, and hence $T$ and $T^{-1}$ are elements of the quotient field of $R$. Since $(1-YT)T^n\in R$ and $Y(T^{-1})^n\in R$ for every $n\in\N_0$, we infer that $\{T,T^{-1}\}\subset\widehat{R}$. Clearly, $K[Y]\subset R\subset\widehat{R}$, and thus $K[Y,T,T^{-1}]\subset\widehat{R}$. Since $S^{-1}=\{s^{-1}\mid s\in S\}\subset R\subset\widehat{R}$, this implies that $S^{-1}(K[Y,T,T^{-1}])\subset\widehat{R}$.

Since $Y$ and $T$ are algebraically independent over $K$, it follows that $K[Y,T]$ is factorial. Note that $K[Y,T,T^{-1}]$ is a quotient overring of $K[Y,T]$, and hence $K[Y,T,T^{-1}]$ is factorial. We infer that $S^{-1}(K[Y,T,T^{-1}])$ is factorial. Moreover, since $R\subset S^{-1}(K[Y,T,T^{-1}])$ and $S^{-1}(K[Y,T,T^{-1}])$ is completely integrally closed, we have that $\widehat{R}\subset S^{-1}(K[Y,T,T^{-1}])$. This implies that $\widehat{R}=S^{-1}(K[Y,T,T^{-1}])$ is factorial, and thus $\widehat{R}$ is a BF-domain. Since $R$ is archimedean, it follows that $\widehat{R}^{\times}\cap R=R^{\times}$, and hence $R$ is a BF-domain by \cite[Corollary 1.3.3]{Ge-HK06a}.
\end{example}

\section{Monoids of ideals and half-factoriality}\label{4}
\smallskip

In this section we study, for a finitary ideal system $r$ of a cancellative monoid $H$, algebraic and arithmetic properties of the {\it semigroup $\mathcal I_r(H)$ of $r$-ideals} and of the {\it semigroup $\mathcal I_r^*(R)$ of $r$-invertible $r$-ideals}. A focus is on the question when these monoids of $r$-ideals are half-factorial (other arithmetical properties of $\mathcal I_r^*(H)$, such as radical factoriality, were recently studied in \cite{Ol-Re20a}).
In Section~\ref{5}, we apply these results to monoids of divisorial ideals and to monoids of usual ring ideals of Mori domains.

Let $H$ be a cancellative monoid and $K$ a quotient group of $H$. An {\it ideal system} on $H$ is a map $r\colon\mathcal P(H)\to\mathcal P(H)$ such that the following conditions are satisfied for all subsets $X,Y\subset H$ and all $c\in H$.
\begin{itemize}
\item $X\subset X_r$.
\item $X\subset Y_r$ implies $X_r\subset Y_r$.
\item $cH\subset\{c\}_r$.
\item $cX_r=(cX)_r$.
\end{itemize}
We refer to \cite{HK98, HK11b} for background on ideal systems. Let $r$ be an ideal system on $H$. A subset $I\subset H$ is called an $r$-ideal if $I_r=I$. Furthermore, a subset $J\subseteq K$ is called a fractional $r$-ideal of $H$ if there is some $c\in H$ such that $cJ$ is an $r$-ideal of $H$. We denote by $\mathcal I_r(H)$ the set of all non-empty $r$-ideals, and we define $r$-multiplication by setting $I\cdot_r J=(IJ)_r$ for all $I,J\in\mathcal I_r(H)$. Then $\mathcal I_r(H)$ together with $r$-multiplication is a reduced semigroup with identity element $H$. Let $\mathcal F_r(H)$ denote the semigroup of non-empty fractional $r$-ideals, $\mathcal F_r(H)^{\times}$ the group of $r$-invertible fractional $r$-ideals, and $\mathcal I_r^*(H)=\mathcal F_r^{\times}(H)\cap\mathcal I_r(H)$ the cancellative monoid of $r$-invertible $r$-ideals of $H$ with $r$-multiplication. We denote by $\mathfrak X(H)$ the set of all non-empty minimal prime $s$-ideals of $H$, by $r$-$\spec(H)$ the set of all prime $r$-ideals of $H$, and by $r$-$\max(H)$ the set of all maximal $r$-ideals of $H$.
We say that $r$ is {\it finitary} if $X_r=\cup E_r$, where the union is taken over all finite subsets $E\subset X$.
For a subset $X\subset\mathsf q(H)$, we set
\[
X_s=XH,\ X_v=(X^{-1})^{-1}\quad\text{and }\quad X_t=\bigcup_{E\subset X,|E|<\infty} E_v\,.
\]
We will use the $s$-system, the $v$-system, and the $t$-system. For every ideal system $r$, we have $X_r\subset X_v$, and if $r$ is finitary, then $X_r\subset X_t$ for all $X\subset H$. We say that $H$ has {\it finite $r$-character} if each $x\in H$ is contained in only finitely many maximal $r$-ideals of $H$.

Let $R$ be a domain with quotient field $K$ and $r$ an ideal system on $R$ (clearly, $R^{\bullet}$ is a monoid and $r$ restricts to an ideal system $r'$ on $R^{\bullet}$ whence for every subset $I\subset R$ we have $I_r=(I^{\bullet})_{r'}\cup\{0\}$). We denote by $\mathcal I_r(R)$ the semigroup of non-zero $r$-ideals of $R$ and $\mathcal I_r^*(R)\subset\mathcal I_r(R)$ is the subsemigroup of $r$-invertible $r$-ideals of $R$. The usual ring ideals form an ideal system, called the $d$-system, and for these ideals we omit all suffices (i.e., $\mathcal I(R)=\mathcal I_d(R)$ and $\mathcal I^*(R)=\mathcal I_d^*(R)$).
For the following equivalent statements let $r$ be an ideal system on $R$ such that every $r$-ideal of $R$ is an ideal of $R$. We say that $R$ is a {\it Cohen-Kaplansky} domain if one of the following equivalent statements hold (\cite[Theorem 4.3]{An-Mo92} and \cite[Proposition 4.5]{Ge-Re19d}).
\begin{itemize}
\item[(a)] $R$ is atomic and has only finitely many atoms up to associates.
\item[(b)] $\mathcal I_r(R)$ is a finitely generated semigroup for some ideal system $r$ on $R$.
\item[(c)] $\mathcal I^*_r(R)$ is a finitely generated semigroup for some ideal system $r$ on $R$.
\item[(d)] $\overline R$ is a semilocal principal ideal domain, $\overline R/(R\DP\overline R)$ is finite, and $|\max(R)|= |\max(\overline R)|$.
\end{itemize}
Thus Corollary~\ref{3.2}.2 and Property (d) imply that a Cohen-Kaplansky domain $R$ is stable if and only $R\subset\overline R$ is a quadratic extension and $R$ has at most two maximal ideals.

\smallskip
\begin{lemma}\label{4.1}
Let $H$ be a cancellative monoid and let $r$ be a finitary ideal system on $H$ such that $\bigcap_{n\in\N_0} (\mathfrak m^n)_r=\emptyset$ for every $\mathfrak m\in r$-$\max(H)$. Then $\mathcal{I}_r(H)$ is unit-cancellative and if $H$ is of finite $r$-character, then $\mathcal{I}_r(H)$ is a \BF-monoid.
\end{lemma}

\begin{proof}
Let $I,J\in\mathcal{I}_r(H)$ be such that $(IJ)_r=I$. Assume that $J$ is proper. Then $J\subset\mathfrak m$ for some $\mathfrak m\in r$-$\max(H)$. It follows by induction that $(IJ^n)_r=I$ for all $n\in\N_0$, and hence $I\subset\bigcap_{n\in\N_0} (J^n)_r\subset\bigcap_{n\in\N_0} (\mathfrak m^n)_r$. Therefore, $I=\emptyset$, a contradiction. Consequently, $\mathcal{I}_r(H)$ is unit-cancellative.

Now let $H$ be of finite $r$-character. We have to show that $\mathcal{I}_r(H)$ is a BF-monoid.

First we show that $\mathcal{I}_r(H)$ is atomic. Since $\mathcal{I}_r(H)$ is unit-cancellative it remains to show by \cite[Lemma 3.1(1)]{Fa-Ge-Ka-Tr17} that $\mathcal{I}_r(H)$ satisfies the ACCP. Assume that $\mathcal{I}_r(H)$ does not satisfy the ACCP. Then there is a sequence $(I_i)_{i=0}^{\infty}$ of elements of $\mathcal{I}_r(H)$ such that $I_i\mathcal{I}_r(H)\subsetneq I_{i+1}\mathcal{I}_r(H)$ for all $i\in\N_0$. Consequently, there is some sequence $(J_i)_{i=0}^{\infty}$ of proper elements of $\mathcal{I}_r(H)$ such that $I_i=(I_{i+1}J_i)_r$ for all $i\in\N_0$. Note that $I_0\subset J_i$ for all $i\in\N_0$. Since $\{\mathfrak m\in r$-$\max(H)\mid I_0\subset\mathfrak m\}$ is finite, there is some $\mathfrak m\in r$-$\max(H)$ such that $\{i\in\N_0\mid J_i\subset\mathfrak m\}$ is infinite. By restricting to a suitable subsequence of $(I_i)_{i\in\N_0}$, we can therefore assume that $J_i\subset\mathfrak m$ for all $i\in\N_0$. Note that $I_0=(I_n\prod_{i=0}^{n-1} J_i)_r$ for every $n\in\N_0$, and thus $I_0\subset (\prod_{i=0}^{n-1} J_i)_r\subset (\mathfrak m^n)_r$ for every $n\in\N_0$. This implies that $I_0\subset\bigcap_{n\in\N_0}(\mathfrak m^n)_r$, and thus $I_0=\emptyset$, a contradiction.

Finally, we prove that $\mathsf L(N)$ is finite for each $N\in\mathcal{I}_r(H)$. Let $N\in\mathcal{I}_r(H)$ and set $\mathcal{M}=\{\mathfrak m\in r$-$\max(H)\mid N\subset\mathfrak m\}$. Observe that $\mathcal{M}$ is finite. For each $\mathfrak{m}\in\mathcal{M}$ set $g_{\mathfrak{m}}=\max\{\ell\in\N\mid N\subset (\mathfrak{m}^{\ell})_r\}$. It is sufficient to show that $n\leq\sum_{\mathfrak m\in\mathcal{M}} g_{\mathfrak m}$ for each $n\in\mathsf L(N)$. Let $n\in\mathsf L(N)$. Clearly, there is a finite sequence $(A_i)_{i=1}^n$ of atoms of $\mathcal{I}_r(H)$ such that $N=(\prod_{i=1}^n A_i)_r$. Since $[1,n]=\bigcup_{\mathfrak m\in\mathcal{M}}\{i\in [1,n]\mid A_i\subset\mathfrak m\}$, we infer that $n\leq\sum_{\mathfrak m\in\mathcal{M}} |\{i\in [1,n]\mid A_i\subset\mathfrak m\}|\leq\sum_{\mathfrak m\in\mathcal{M}} g_{\mathfrak m}$.
\end{proof}

Let $H$ be a cancellative monoid and $r$ a finitary ideal system on $H$. Observe that if $H$ is strictly $r$-noetherian (for the definition of strictly $r$-noetherian monoids we refer to \cite[8.4 Definition, page 87]{HK98}), then it follows from \cite[9.1 Theorem, page 94]{HK98} that $\bigcap_{n\in\N_0} (\mathfrak m^n)_r=\emptyset$ for every $\mathfrak m\in r$-$\max(H)$. Also note that if $H$ is a Mori monoid and $r$-$\max(H)=\mathfrak{X}(H)$, then $H$ is of finite $r$-character (this is an easy consequence of \cite[Theorem 2.2.5.1]{Ge-HK06a}).

\begin{proposition}\label{4.2}
Let $H$ be a finitely primary monoid of rank one, $\mathfrak m=H\setminus H^{\times}$, $\mathfrak q=\widehat{H}\setminus\widehat{H}^{\times}$, and let $r$ be a finitary ideal system on $H$.
\begin{enumerate}
\item The following statements are equivalent.
\begin{enumerate}
\item[(a)] $H$ is half-factorial.
\item[(b)] $u\widehat{H}=v\widehat{H}$ for all $u,v\in\mathcal{A}(H)$.
\item[(c)] $u\widehat{H}=\mathfrak q$ for all $u\in\mathcal{A}(H)$.
\end{enumerate}
\item The following statements are equivalent.
\begin{enumerate}
\item[(a)] $\mathcal{I}_r(H)$ is half-factorial.
\item[(b)] $A\widehat{H}=B\widehat{H}$ for all $A,B\in\mathcal{A}(\mathcal{I}_r(H))$.
\item[(c)] $A\widehat{H}=\mathfrak q$ for all $A\in\mathcal{A}(\mathcal{I}_r(H))$.
\item[(d)] If $k\in\N$ and $A_i\in\mathcal{A}(\mathcal{I}_r(H))$ for every $i\in [1,k]$, then $\prod_{i=1}^k A_i\not\subset (\mathfrak m^{k+1})_r$.
\item[(e)] $H$ is half-factorial and for every nonprincipal $A\in\mathcal{A}(\mathcal{I}_r(H))$ it follows that $A\not\subset (\mathfrak m^2)_r$.
\end{enumerate}
\end{enumerate}
\end{proposition}

\begin{proof}
Since $H$ is finitely primary of rank one, there is some $q\in\mathfrak q$ such that $\mathfrak q=q\widehat{H}$.

\smallskip
1.(a) $\Rightarrow$ 1.(b) Let $u,v\in\mathcal{A}(H)$. There are some $k,\ell\in\N$ such that $u\widehat{H}=q^k\widehat{H}$ and $v\widehat{H}=q^{\ell}\widehat{H}$. It follows that $u^{\ell}\widehat{H}=v^k\widehat{H}$, and hence $u^{\ell}=v^k\varepsilon$ for some $\varepsilon\in\widehat{H}^{\times}$. Moreover, there is some $a\in (H\DP\widehat{H})$. Since $\mathsf L_H(a\varepsilon^n)\subset [0,{\rm v}_q(a)]$ for every $n\in\N_0$, there are some $n_1,n_2\in\N_0$ such that $n_1<n_2$ and $\mathsf L_H(a\varepsilon^{n_1})=\mathsf L_H(a\varepsilon^{n_2})$. Set $b=a\varepsilon^{n_1}$ and set $n=n_2-n_1$. Then $n\in\N$, $b,b\varepsilon^n\in H$ and $\mathsf L_H(b)=\mathsf L_H(b\varepsilon^n)$. There is some $h\in\N_0$ such that $\mathsf L_H(b)=\{h\}$. Note that $u^{n\ell}b=v^{nk}b\varepsilon^n$, and hence $\{n\ell+h\}=\mathsf L_H(u^{n\ell}b)=\mathsf L_H(v^{nk}b\varepsilon^n)=\{nk+h\}$. We infer that $\ell=k$, and hence $u\widehat{H}=q^k\widehat{H}=q^{\ell}\widehat{H}=v\widehat{H}$.

\smallskip
1.(b) $\Rightarrow$ 1.(c) Since $(H\DP\widehat{H})\not=\emptyset$, there is some $m\in\N$ such that $q^m,q^{m+1}\in H$. There is some $\ell\in\N$ such that $u\widehat{H}=q^{\ell}\widehat{H}$ for all $u\in\mathcal{A}(H)$. There are some $a,b\in\N$ such that $q^m$ is a product of $a$ atoms of $H$ and $q^{m+1}$ is a product of $b$ atoms of $H$. We infer that $q^m\widehat{H}=q^{a\ell}\widehat{H}$ and $q^{m+1}\widehat{H}=q^{b\ell}\widehat{H}$. This implies that $b\ell=m+1=a\ell+1$, and hence $\ell=1$ and $u\widehat{H}=\mathfrak q$.

\smallskip
1.(c) $\Rightarrow$ 1.(a) Let $k,\ell\in\N_0$, let $u_i\in\mathcal{A}(H)$ for every $i\in [1,k]$ and let $v_j\in\mathcal{A}(H)$ for every $j\in [1,\ell]$ be such that $\prod_{i=1}^k u_i=\prod_{j=1}^{\ell} v_j$. Then $\mathfrak q^k=\prod_{i=1}^k u_i\widehat{H}=\prod_{j=1}^{\ell} v_j\widehat{H}=\mathfrak q^{\ell}$. Consequently, $k=\ell$.

\smallskip
2. Note that $H$ is strongly primary and $r$-$\max(H)=\{\mathfrak m\}$. Therefore, $\bigcap_{n\in\N_0} (\mathfrak m^n)_r=\emptyset$. We infer by Lemma~\ref{4.1} that $\mathcal{I}_r(H)$ is a unit-cancellative atomic monoid. Since $H$ is $r$-local, we have that $\mathcal{A}(\mathcal{I}^*_r(H))=\{uH\mid u\in\mathcal{A}(H)\}$. Moreover, $\mathcal{I}^*_r(H)$ is a divisor-closed submonoid of $\mathcal{I}_r(H)$. Therefore, $\{uH\mid u\in\mathcal{A}(H)\}\subset\mathcal{A}(\mathcal{I}_r(H))$. Note that if $I$ is a non-empty $s$-ideal of $H$, then $I\widehat{H}=I_r\widehat{H}$ (since $I\widehat{H}=q^m\widehat{H}$ for some $m\in\N_0$, it follows that $q^m\widehat{H}=I\widehat{H}\subset I_r\widehat{H}\subset I_t\widehat{H}\subset (I\widehat{H})_t=(q^m\widehat{H})_t=q^m(\widehat{H})_t=q^m\widehat{H}$).

\smallskip
2.(a) $\Rightarrow$ 2.(b) Let $A,B\in\mathcal{A}(\mathcal{I}_r(H))$. There are some $k,\ell\in\N$ such that $A\widehat{H}=\mathfrak q^k$ and $B\widehat{H}=\mathfrak q^{\ell}$. This implies that $A^{\ell}\widehat{H}=B^k\widehat{H}$, and hence $(A^{\ell}(H\DP\widehat{H}))_r=(B^k(H\DP\widehat{H}))_r$. Since $(H\DP\widehat{H})$ is a non-empty $r$-ideal of $H$, there is some $m\in\N_0$ such that $\mathsf L((H\DP\widehat{H}))=\{m\}$. Therefore, $\{\ell+m\}=\mathsf L((A^{\ell}(H\DP\widehat{H}))_r)=\mathsf L((B^k(H\DP\widehat{H}))_r)=\{k+m\}$, and thus $\ell=k$. We infer that $A\widehat{H}=\mathfrak q^k=\mathfrak q^{\ell}=B\widehat{H}$.

\smallskip
2.(b) $\Rightarrow$ 2.(c) Since $(H\DP\widehat{H})\not=\emptyset$, there is some $m\in\N$ such that $q^m,q^{m+1}\in H$. There is some $\ell\in\N$ such that $A\widehat{H}=q^{\ell}\widehat{H}$ for all $A\in\mathcal{A}(\mathcal{I}_r(H))$. Since $\mathcal{I}_r(H)$ is atomic, there are some $a,b\in\N$ such that $q^mH$ is an $r$-product of $a$ atoms of $\mathcal{I}_r(H)$ and $q^{m+1}H$ is an $r$-product of $b$ atoms of $\mathcal{I}_r(H)$. This implies that $q^m\widehat{H}=q^{a\ell}\widehat{H}$ and $q^{m+1}\widehat{H}=q^{b\ell}\widehat{H}$. Therefore, $b\ell=m+1=a\ell+1$, and hence $\ell=1$ and $A\widehat{H}=\mathfrak q$.

\smallskip
2.(c) $\Rightarrow$ 2.(a) Let $k,\ell\in\N_0$, let $A_i\in\mathcal{A}(\mathcal{I}_r(H))$ for every $i\in [1,k]$ and let $B_j\in\mathcal{A}(\mathcal{I}_r(H))$ for every $j\in [1,\ell]$ be such that $(\prod_{i=1}^k A_i)_r=(\prod_{j=1}^{\ell} B_j)_r$. Then $\mathfrak q^k=(\prod_{i=1}^k A_i)_r\widehat{H}=(\prod_{j=1}^{\ell} B_j)_r\widehat{H}=\mathfrak q^{\ell}$. Therefore, $k=\ell$.

\smallskip
2.(c) $\Rightarrow$ 2.(d) Let $k\in\N$ and let $A_i\in\mathcal{A}(\mathcal{I}_r(H))$ for every $i\in [1,k]$. Assume that $\prod_{i=1}^k A_i\subset (\mathfrak m^{k+1})_r$. Note that $\mathfrak m\in\mathcal{A}(\mathcal{I}_r(H))$ (since $\mathcal{I}_r(H)$ is unit-cancellative). Therefore, $A_i\widehat{H}=\mathfrak m\widehat{H}=\mathfrak q$ for all $i\in [1,k]$, and thus $\mathfrak q^k=(\prod_{i=1}^k A_i)_r\widehat{H}\subset (\mathfrak m^{k+1})_r\widehat{H}=\mathfrak q^{k+1}$, a contradiction.

\smallskip
2.(d) $\Rightarrow$ 2.(e) It remains to show that $H$ is half-factorial. Let $k,\ell\in\N$, let $u_i\in\mathcal{A}(H)$ for every $i\in [1,k]$ and let $v_j\in\mathcal{A}(H)$ for every $j\in [1,\ell]$ be such that $\prod_{i=1}^k u_i=\prod_{j=1}^{\ell} v_j$. Observe that $u_iH,v_jH\in\mathcal{A}(\mathcal{I}_r(H))$ for all $i\in [1,k]$ and $j\in [1,\ell]$. We infer that $\prod_{i=1}^k u_i\not\in (\mathfrak m^{\ell+1})_r$ and $\prod_{j=1}^{\ell} v_j\not\in (\mathfrak m^{k+1})_r$. Therefore, $k<\ell+1$ and $\ell<k+1$, and hence $k=\ell$.

\smallskip
2.(e) $\Rightarrow$ 2.(c) Let $A\in\mathcal{A}(\mathcal{I}_r(H))$.

\smallskip
Case 1. $A$ is principal. Then $A=uH$ for some $u\in\mathcal{A}(H)$. By 1. we have that $A\widehat{H}=u\widehat{H}=\mathfrak q$.

\smallskip
Case 2. $A$ is not principal. Then $A\not\subset (\mathfrak m^2)_r$, and hence there is some $v\in A\setminus (\mathfrak m^2)_r$. Observe that $v\in\mathcal{A}(H)$. It follows from 1. that $\mathfrak q=v\widehat{H}\subset A\widehat{H}\subset\mathfrak q$, and thus $A\widehat{H}=\mathfrak q$.
\end{proof}

Observe that some of the semigroups (e.g. $\mathcal{I}_r(H)$) in the following result may not always be unit-cancellative. In that case, we apply the original definitions for being an atom or being half-factorial to commutative semigroups with identity (which are not necessarily unit-cancellative).

\begin{proposition}\label{4.3}
Let $H$ be a cancellative monoid and $r$ be a finitary ideal system on $H$ such that $H$ is of finite $r$-character and $r$-$\max(H)=\mathfrak X(H)$.
\begin{enumerate}
\item $\mathcal{I}_r(H)\cong\coprod_{\mathfrak m\in\mathfrak X(H)}\mathcal{I}_{r_{\mathfrak m}}(H_{\mathfrak m})$ and $\mathcal{I}^*_r(H)\cong\coprod_{\mathfrak m\in\mathfrak X(H)}\mathcal{I}^*_{r_{\mathfrak m}}(H_{\mathfrak m})$.
\item $\mathcal{I}_r(H)$ is half-factorial if and only if $\mathcal{I}_{r_{\mathfrak m}}(H_{\mathfrak m})$ is half-factorial for every $\mathfrak m\in\mathfrak X(H)$ and $\mathcal{I}^*_r(H)$ is half-factorial if and only if $H_{\mathfrak m}$ is half-factorial for every $\mathfrak m\in\mathfrak X(H)$.
\item If $A\in\mathcal{A}(\mathcal{I}_r(H))$, then $\sqrt{A}\in\mathfrak X(H)$.
\item For every $\mathfrak m\in\mathfrak X(H)$ we have that $\mathcal{A}(\mathcal{I}_{r_{\mathfrak m}}(H_{\mathfrak m}))=\{A_{\mathfrak m}\mid A\in\mathcal{A}(\mathcal{I}_r(H)),A\subset\mathfrak m\}$.
\end{enumerate}
\end{proposition}

\begin{proof}
Claim: For every $I\in\mathcal{I}_r(H)$ it follows that $I=(\prod_{\mathfrak q\in\mathfrak X(H)} (I_{\mathfrak q}\cap H))_r$.

Proof of the claim: Let $I\in\mathcal{I}_r(H)$. Since $H$ is of finite $r$-character, it follows that $I_{\mathfrak q}\cap H=H$ for all but finitely many $\mathfrak q\in\mathfrak X(H)$. Note that if $\mathfrak q\in\mathfrak X(H)$ and $I\subset\mathfrak q$, then $I_{\mathfrak q}\cap H$ is a $\mathfrak q$-primary $r$-ideal of $H$, and $(I_{\mathfrak q}\cap H)_{\mathfrak q}=I_{\mathfrak q}$. Therefore, $((\prod_{\mathfrak q\in\mathfrak X(H)} (I_{\mathfrak q}\cap H))_r)_{\mathfrak m}=(\prod_{\mathfrak q\in\mathfrak X(H)} (I_{\mathfrak q}\cap H)_{\mathfrak m})_{r_{\mathfrak m}}=I_{\mathfrak m}$. Consequently, $I=(\prod_{\mathfrak q\in\mathfrak X(H)} (I_{\mathfrak q}\cap H))_r$.

1. Let $f\colon\mathcal{I}_r(H)\rightarrow\coprod_{\mathfrak m\in\mathfrak X(H)}\mathcal{I}_{r_{\mathfrak m}}(H_{\mathfrak m})$ be defined by $f(I)=(I_{\mathfrak m})_{\mathfrak m\in\mathfrak X(H)}$ for every $I\in\mathcal{I}_r(H)$. Since $H$ is of finite $r$-character it is clear that $f$ is well-defined. It is straightforward to show that $f$ is a monoid homomorphism. If $I,J\in\mathcal{I}_r(H)$ are such that $I_{\mathfrak m}=J_{\mathfrak m}$ for all $\mathfrak m\in\mathfrak X(H)$, then $I=\bigcap_{\mathfrak m\in r\textnormal{-}\max(H)} I_{\mathfrak m}=\bigcap_{\mathfrak m\in r\textnormal{-}\max(H)} J_{\mathfrak m}=J$. Therefore, $f$ is injective. It remains to show that $f$ is surjective. Let $(I_{\mathfrak m})_{\mathfrak m\in\mathfrak X(H)}\in\coprod_{\mathfrak m\in\mathfrak X(H)}\mathcal{I}^*_{r_{\mathfrak m}}(H_{\mathfrak m})$. Set $I=(\prod_{\mathfrak m\in\mathfrak X (H)} (I_{\mathfrak m}\cap H))_r$. Then $I\in\mathcal{I}_r(H)$ and $(I_{\mathfrak q}\cap H)_{\mathfrak q}=I_{\mathfrak q}$ for every $\mathfrak q\in\mathfrak X(H)$. Therefore, $f$ is surjective. If $I\in\mathcal{I}^*_r(H)$, then $I_{\mathfrak m}\in\mathcal{I}^*_{r_{\mathfrak m}}(H_{\mathfrak m})$ for every $\mathfrak m\in\mathfrak X(H)$, and thus $f\mid_{\mathcal{I}^*_r(H)}\colon\mathcal{I}^*_r(H)\rightarrow\coprod_{\mathfrak m\in\mathfrak X(H)}\mathcal{I}^*_{r_{\mathfrak m}}(H_{\mathfrak m})$ is a monoid isomorphism.

\smallskip
2. It is an immediate consequence of 1. that $\mathcal{I}_r(H)$ is half-factorial if and only if $\mathcal{I}_{r_{\mathfrak m}}(H_{\mathfrak m})$ is half-factorial for every $\mathfrak m\in\mathfrak X(H)$ and $\mathcal{I}^*_r(H)$ is half-factorial if and only if $\mathcal{I}^*_{r_{\mathfrak m}}(H_{\mathfrak m})$ is half-factorial for every $\mathfrak m\in\mathfrak X(H)$. Note that if $\mathfrak m\in\mathfrak X(H)$, then $H_{\mathfrak m}$ is $r_{\mathfrak m}$-local, and hence $\mathcal{I}^*_{r_{\mathfrak m}}(H_{\mathfrak m})=\{xH_{\mathfrak m}\mid x\in H_{\mathfrak m}^{\bullet}\}$. Clearly, $\{xH_{\mathfrak m}\mid x\in H_{\mathfrak m}^{\bullet}\}\cong (H_{\mathfrak m}^{\bullet})_{\rm red}$ is half-factorial if and only if $H_{\mathfrak m}$ is half-factorial.

\smallskip
3. Let $A\in\mathcal{A}(\mathcal{I}_r(H))$. Then $A\subset\mathfrak m$ for some $\mathfrak m\in\mathfrak X(H)$. Set $J=(\prod_{\mathfrak q\in\mathfrak X(H)\setminus\{\mathfrak m\}} (A_{\mathfrak q}\cap H))_r$. We infer by the claim that $A=(J(A_{\mathfrak m}\cap H))_r$. Since $A_{\mathfrak m}\cap H$ is a proper $r$-ideal of $H$ this implies that $A=A_{\mathfrak m}\cap H$. Since $A_{\mathfrak m}$ is $\mathfrak m_{\mathfrak m}$-primary, we have that $A_{\mathfrak m}\cap H$ is $\mathfrak m$-primary, and thus $\sqrt{A}=\mathfrak m$.

\smallskip
4. Let $\mathfrak m\in\mathfrak X(H)$. First let $B\in\mathcal{A}(\mathcal{I}_{r_{\mathfrak m}}(H_{\mathfrak m}))$. Set $A=B\cap H$. Then $A$ is a proper $r$-ideal of $H$ and $B=A_{\mathfrak m}$. It remains to show that $A\in\mathcal{I}_r(H)$. Let $I,J\in\mathcal{I}_r(H)$ be such that $A=(IJ)_r$. Then $B=(I_{\mathfrak m}J_{\mathfrak m})_{r_{\mathfrak m}}$, and hence $I_{\mathfrak m}=H_{\mathfrak m}$ or $J_{\mathfrak m}=H_{\mathfrak m}$. Without restriction let $I_{\mathfrak m}=H_{\mathfrak m}$. Then $I\not\subset\mathfrak m$. Since $A$ is $\mathfrak m$-primary and $A\subset I$, this implies that $I=H$.

Now let $B\in\mathcal{A}(\mathcal{I}_r(H))$ be such that $B\subset\mathfrak m$. Let $I,J\in\mathcal{I}_{r_{\mathfrak m}}(H_{\mathfrak m})$ be such that $B_{\mathfrak m}=(IJ)_{r_{\mathfrak m}}$. It is straightforward to check $r$-locally that $B=((I\cap H)(J\cap H))_r$. Note that $I\cap H,J\cap H\in\mathcal{I}_r(H)$, and hence $I\cap H=H$ or $J\cap H=H$. Without restriction let $I\cap H=H$. Consequently, $I=H_{\mathfrak m}$.
\end{proof}

\begin{theorem}\label{4.4}
Let $H$ be a cancellative monoid and let $r$ be a finitary ideal system on $H$ such that $H$ is of finite $r$-character and $H_{\mathfrak m}$ is finitely primary for every $\mathfrak m\in r$-$\max(H)$. Then $\mathcal{I}_r(H)$ is half-factorial if and only if $\mathcal{I}^*_r(H)$ is half-factorial and for every $A\in\mathcal{A}(\mathcal{I}_r(H))\setminus\mathcal{I}^*_r(H)$ we have that $A\not\subset ((\sqrt{A})^2)_r$.
\end{theorem}

\begin{proof} First let $\mathcal{I}_r(H)$ be half-factorial. Since $\mathcal{I}^*_r(H)$ is a divisor-closed submonoid of $\mathcal{I}_r(H)$ we have that $\mathcal{I}^*_r(H)$ is half-factorial. Let $\mathfrak m\in r$-$\max(H)$. It follows by Proposition~\ref{4.3}.2 that $\mathcal{I}_{r_{\mathfrak m}}(H_{\mathfrak m})$ and $H_{\mathfrak m}$ are half-factorial. Therefore, $H_{\mathfrak m}$ is finitely primary of rank one by \cite[Theorem 3.1.5]{Ge-HK06a}. We infer by Proposition~\ref{4.2}.2 that for every nonprincipal $B\in\mathcal{A}(\mathcal{I}_{r_{\mathfrak m}}(H_{\mathfrak m}))$ we have that $B\not\subset (\mathfrak m_{\mathfrak m}^2)_{r_{\mathfrak m}}$. We infer by Proposition~\ref{4.3}.2 that $\mathcal{I}^*_r(H)$ is half-factorial. Let $A\in\mathcal{A}(\mathcal{I}_r(H))\setminus\mathcal{I}^*_r(H)$. Then $\sqrt{A}\in r$-$\max(H)$ by Proposition~\ref{4.3}.3. Without restriction let $\sqrt{A}=\mathfrak m$. It follows by Proposition~\ref{4.3} that $A_{\mathfrak m}\in\mathcal{A}(\mathcal{I}_{r_{\mathfrak m}}(H_{\mathfrak m}))$. If $A_{\mathfrak m}$ is a principal ideal of $H_{\mathfrak m}$, then $A$ is $r$-locally principal, and since $H$ is of finite $r$-character, $A$ is $r$-invertible, a contradiction. Therefore, $A_{\mathfrak m}$ is not a principal ideal of $H_{\mathfrak m}$ and $A_{\mathfrak m}\not\subset (\mathfrak m_{\mathfrak m}^2)_{r_{\mathfrak m}}$. Since $A$ and $(\mathfrak m^2)_r$ are $\mathfrak m$-primary this implies that $A\not\subset (\mathfrak m^2)_r$.

Now let $\mathcal{I}^*_r(H)$ be half-factorial and let for every $A\in\mathcal{A}(\mathcal{I}_r(H))\setminus\mathcal{I}^*_r(H)$, $A\not\subset ((\sqrt{A})^2)_r$. Let $\mathfrak m\in r$-$\max(H)$. It follows from Proposition~\ref{4.3}.2 that $H_{\mathfrak m}$ is half-factorial. Consequently, $H_{\mathfrak m}$ is finitely primary of rank one. Let $B\in\mathcal{A}(\mathcal{I}_{r_{\mathfrak m}}(H_{\mathfrak m}))$ be not principal. Then $B=A_{\mathfrak m}$ for some $A\in\mathcal{A}(\mathcal{I}_r(H))$ with $A\subset\mathfrak m$ by Proposition~\ref{4.3}.4. It follows from Proposition~\ref{4.3}.3 that $\sqrt{A}=\mathfrak m$. Obviously, $A$ is not $r$-invertible. Therefore, $A\not\subset (\mathfrak m^2)_r$. Since $A$ and $(\mathfrak m^2)_r$ are $\mathfrak m$-primary we have that $B\not\subset (\mathfrak m_{\mathfrak m}^2)_{r_{\mathfrak m}}$. We infer by Proposition~\ref{4.2}.2 that $\mathcal{I}_{r_{\mathfrak m}}(H_{\mathfrak m})$ is half-factorial.
\end{proof}

\begin{corollary}\label{4.5}
Let $H$ be a cancellative monoid and let $r$ be a finitary ideal system on $H$ such that $H$ is of finite $r$-character and $H_{\mathfrak m}$ is finitely primary and $\mathfrak m^2$ is contained in some proper $r$-invertible $r$-ideal of $H$ for every $\mathfrak m\in r$-$\max(H)$. Then $\mathcal{I}_r(H)$ is half-factorial if and only if $\mathcal{I}^*_r(H)$ is half-factorial.
\end{corollary}

\begin{proof}
By Theorem~\ref{4.4} it is sufficient to show that for every $A\in\mathcal{A}(\mathcal{I}_r(H))\setminus\mathcal{I}^*_r(H)$, we have that $A\not\subset ((\sqrt{A})^2)_r$. Let $A\in\mathcal{A}(\mathcal{I}_r(H))\setminus\mathcal{I}^*_r(H)$. Assume that $A\subset ((\sqrt{A})^2)_r$. There is some $\mathfrak m\in r$-$\max(H)$ such that $A\subset\mathfrak m$. We infer that $\mathfrak m^2\subset I$ for some proper $I\in\mathcal{I}^*_r(H)$. Since $\sqrt{A}\subset\mathfrak m$, it follows that $A\subset ((\sqrt{A})^2)_r\subset (\mathfrak m^2)_r\subset I$, and thus $A=I\in\mathcal{I}^*_r(H)$, a contradiction.
\end{proof}

\begin{lemma}\label{4.6}
Let $L$ be a finite field, let $K\subset L$ be a subfield, let $X$ be an indeterminate over $L$ and let $R=K+XL\LK X\RK$. Then $R$ is a local Cohen-Kaplansky domain with maximal ideal $XL\LK X\RK$ and $R$ is divisorial if and only if $[L\DP K]\leq 2$.
\end{lemma}

\begin{proof}
It is an immediate consequence of \cite[Corollary 7.2]{An-Mo92} that $R$ is a local Cohen-Kaplansky domain with maximal ideal $XL\LK X\RK$. Set $\mathfrak m=XL\LK X\RK$. Without restriction let $K\not=L$. Then $\mathfrak m^{-1}=(\mathfrak m\DP\mathfrak m)=L\LK X\RK$. Since $R$ is a local one-dimensional noetherian domain we have by \cite[Theorem 3.8]{Ma68a} that $R$ is divisorial if and only if $L\LK X\RK$ is a $2$-generated $R$-module. For $h\in L\LK X\RK$ let $h_0$ denote the constant term of $h$.

If $L\LK X\RK$ is a $2$-generated $R$-module, then $L\LK X\RK=\langle f,g\rangle_R$, whence $L=\langle f_0,g_0\rangle_K$, and so $[L\DP K]=2$. Conversely, let $[L\DP K]=2$. Then $L=\langle 1,a\rangle_K$ for some $a\in L$. Observe that $L\LK X\RK=\langle 1,a\rangle_R$.
\end{proof}

\begin{example}\label{4.7}
Let $L$ be a finite field, let $K\subset L$ be a subfield, let $n\in\N_{\geq 2}$ and let $R=K+X^nL\LK X\RK$. Then $R$ is a local Cohen-Kaplansky domain with maximal ideal $X^nL\LK X\RK$, $R$ is not half-factorial and the square of the maximal ideal of $R$ is contained in a proper principal ideal of $R$.
\end{example}

\begin{proof}
By \cite[Corollary 7.2]{An-Mo92} we have that $R$ is a local Cohen-Kaplansky domain with maximal ideal $X^nL\LK X\RK$ such that $R$ is not half-factorial. Set $\mathfrak m=X^nL\LK X\RK$. Then $\mathfrak m^2=X^{2n}L\LK X\RK\subset X^nR$ and $X^nR$ is a proper principal ideal of $R$.
\end{proof}

\section{Arithmetic of stable orders in Dedekind domains}\label{5}
\smallskip

In this section we derive the main arithmetical results of the paper. For monoids of ideals of stable Mori domains, we study the catenary degree, the monotone catenary degree and we establish characterizations when these monoids are half-factorial and when they are transfer Krull. We demonstrate in remarks and examples that none of the main statements in Theorems~\ref{5.9} and~\ref{5.10} hold true without the stability assumption.

We need the concepts of catenary degrees, transfer homomorphisms, and transfer Krull monoids.
Let $H$ be an atomic monoid. The free abelian monoid $\mathsf Z(H)=\mathcal F(\mathcal A(H_{\red}))$ denotes the factorization monoid of $H$ and $\pi\colon\mathsf Z(H)\to H_{\red}$ the canonical epimorphism. For every element $a\in H$,
$\mathsf Z(a)=\pi^{-1} (aH^{\times})$ is the {\it set of factorizations} of $a$. Note that
$\mathsf L(a)=\{|z|\mid z\in\mathsf Z(a)\}\subset\N_0$ is the set of lengths of $a$.
Suppose that $H$ is atomic. If $z,z'\in\mathsf Z(H)$ are two factorizations, say
\[
z=u_1\cdot\ldots\cdot u_{\ell}v_1\cdot\ldots\cdot v_m\quad\text{and}\quad z'=u_1\cdot\ldots\cdot u_{\ell}w_1\cdot\ldots\cdot w_n\,,
\]
where $\ell,m,n\in\N_0$ and all $u_i,v_j,w_k\in\mathcal A(H_{\red})$ such that $v_j\ne w_k$ for all $j\in [1,m]$ and all $k\in [1,n]$, then $\mathsf d(z,z')=\max\{m,n\}$ is the distance between $z$ and $z'$.

Let $a\in H$ and
$N\in\mathbb N_0\cup\{\infty\}$. A finite sequence $z_0,\ldots,z_k\in\mathsf Z(a)$ is called a {\it $($monotone$)$ $N$-chain of factorizations of $a$} if $\mathsf d(z_{i-1},z_i)\le N$ for all $i\in [1,k]$ (and $|z_0|\le\ldots\le |z_k|$ or $|z_0|\ge\ldots\ge |z_k|$). We denote by $\mathsf c(a)$ (or by $\mathsf c_{\mon}(a)$ resp.) the smallest $N\in\N _0\cup\{\infty\}$ such that any two factorizations $z,\,z'\in\mathsf Z(a)$ can be concatenated by an $N$-chain (or by a monotone $N$-chain resp.).
Then
\[
\mathsf c(H)=\sup\{\mathsf c(b)\mid b\in H\}\in\N_0\cup
\{\infty\}\quad\text{and}\quad\mathsf c_{\mon}(H)=\sup\{
\mathsf c_{\mon}(b)\mid b\in H\}\in\N_0\cup\{\infty\}\quad
\,
\]
denote the\ {\it catenary degree}\ and the\ {\it monotone catenary degree} of $H$. By definition, we have $\mathsf c(H)\le\mathsf c_{\mon}(H)$, and $H$ is factorial if and only if $\mathsf c(H)=0$. If $H$ is cancellative but not factorial, then, by \cite[Theorem 1.6.3]{Ge-HK06a},
\begin{equation}\label{basic-inequality-1}
2+\sup\Delta(H)\le\mathsf c(H)\le\mathsf c_{\mon}(H)\,,
\end{equation}
whence $\mathsf c(H)\le 2$ implies that $H$ is half-factorial and that $2=\mathsf c(H)=\mathsf c_{\mon}(H)$. Let
\begin{equation}\label{finitely-primary}
H\subset F=F^{\times}\time\mathcal F(\{p_1,\ldots, p_s\})
\end{equation}
be a finitely primary monoid of rank $s\in\N$ and exponent $\alpha\in\N$. Then, by \cite[Theorem 3.1.5]{Ge-HK06a}, we have
\begin{equation}
\text{If $s=1$, then $\rho(H)\le 2\alpha-1$ and $\mathsf c(H)\le 3\alpha-1$.}\label{s=1}
\end{equation}
\begin{equation}
\text{If $s\ge 2$, then $\rho(H)=\infty$ and $\mathsf c(H)\le 2\alpha+1$.}\label{s=2}
\end{equation}

A monoid homomorphism $\theta\colon H\to B$ between monoids is said to be a {\it transfer homomorphism} if the following two properties are satisfied.
\begin{enumerate}
\item[{\bf (T\,1)\,}] $B=\theta(H) B^\times$\ and\ $\theta^{-1} (B^\times)=H^\times$.
\item[{\bf (T\,2)\,}] If $u\in H$,\ $b,\,c\in B$\ and\ $\theta(u)=bc$, then there exist\ $v,\,w\in H$\ such that\ $u=vw$,\
$\theta(v)\in b B^{\times}$\ and\ $\theta(w)\in c B^{\times}$.
\end{enumerate}
A monoid $H$ is said to be a {\it transfer Krull monoid} if it allows a transfer homomorphism $\theta$ to a Krull monoid $B$.
Since the identity map is a transfer homomorphism, Krull monoids are transfer Krull, but transfer Krull monoids need neither be commutative (though here we restrict to the commutative setting), nor Mori, nor completely integrally closed. The arithmetic of Krull monoids is best understood (compared with various other classes of monoids and domains), and a transfer homomorphism allows to pull back arithmetical properties of the Krull monoid $B$ to the original monoid $H$. We refer to the surveys \cite{Ge16c, Ge-Zh20a} for examples and basic properties of transfer Krull monoids.

All Dedekind domains are transfer Krull and stable. However, there are orders in Dedekind domains that are transfer Krull but not stable (Remark~\ref{5.15}) and there are orders that are stable but not transfer Krull (all orders in quadratic number fields are stable but not all of them are transfer Krull). Half-factorial monoids are trivial examples of transfer Krull monoids (if $H$ is half-factorial, then $\theta\colon H\to (\N_0,+)$, defined by $\theta(u)=1$ for all $u\in\mathcal A(H)$ and $\theta(\varepsilon)=0$ for all $\varepsilon\in H^{\times}$, is a transfer homomorphism). Thus a result (as given in Theorems~\ref{5.1} and~\ref{5.9}), stating that monoids of a given type are transfer Krull if and only if they are half-factorial, means that their arithmetic is different from the arithmetic of Krull monoids and equal only in the trivial case. For recent work on the half-factoriality of transfer Krull monoids we refer to \cite{G-L-T-Z21}.

\smallskip
We start with a result on the finiteness of the catenary degree of weakly Krull Mori domains.

\begin{theorem}\label{5.1}
Let $R$ be a weakly Krull Mori domain.
\begin{enumerate}
\item For every $\mathfrak p\in\mathfrak X(R)$, $\mathsf c(R_{\mathfrak p})<\infty$, and $\rho(R_{\mathfrak p})<\infty$ if and only if $(R_{\mathfrak p}\DP\widehat{R_{\mathfrak p}})\ne\{0\}$ and $\widehat{R_{\mathfrak p}}$ is local.
\item $\mathsf c(\mathcal{I}_v(R))=\sup\{\mathsf c(\mathcal{I}_{v_{\mathfrak p}}(R_{\mathfrak p}))\mid {\mathfrak p}\in\mathfrak X(R)\}$ and $\mathsf c(\mathcal{I}^*_v(R))=\sup\{\mathsf c(\mathcal{I}^*_{v_{\mathfrak p}}(R_{\mathfrak p}))\mid {\mathfrak p}\in\mathfrak X(R)\}$.
\item If $(R\DP\widehat R)\ne\{0\}$, then $\mathsf c(\mathcal{I}^*_v(R))\le\mathsf c(\mathcal{I}_v(R))<\infty$.
\item $\mathcal I_v^*(R)$ is a Mori monoid and it is half-factorial if and only if it is transfer Krull.
\end{enumerate}
\end{theorem}

\begin{proof}
Since $R$ is a weakly Krull Mori domain, we have $t$-$\spec(R)=\mathfrak X(R)$ by \cite[Theorem 24.5]{HK98}. Thus all assumptions of Proposition~\ref{4.3} are satisfied.

1. Let $\mathfrak p\in\mathfrak X(R)$. Since $R_{\mathfrak p}$ is a one-dimensional local Mori domain, it is strongly primary and hence locally tame by \cite[Theorem 3.9]{Ge-Ro20a}. Thus its catenary degree is finite by \cite[Theorem 4.1]{Ge-Go-Tr20}. If $(R_{\mathfrak p}\DP\widehat{R_{\mathfrak p}})=\{0\}$, then $\rho(R_{\mathfrak p})=\infty$ by \cite[Theorem 3.7]{Ge-Ro20a}. Suppose that $(R_{\mathfrak p}\DP\widehat{R_{\mathfrak p}})\ne\{0\}$. Then $R^{\bullet}$ is finitely primary by \cite[Proposition 2.10.7]{Ge-HK06a}, whence the claim on the elasticity follows from \eqref{s=1} and \eqref{s=2}.

2. Since the catenary degree of a coproduct equals the supremum of the individual catenary degrees (\cite[Proposition 1.6.8]{Ge-HK06a}), the assertion follows from Proposition~\ref{4.3}.1.

3. Since $\mathcal{I}^*_v(R)$ is a divisor-closed submonoid of $\mathcal{I}_v(R)$, the inequality between their catenary degrees holds. If $(R\DP\widehat R)\ne\{0\}$,
then almost all $R_{\mathfrak p}$ are discrete valuation domains whence their catenary degree is finite. Thus the claim follows from 2. and from Proposition~\ref{4.3}.1.

4. See \cite[Proposition 7.3]{Ge-Zh20a}.
\end{proof}

There are primary Mori monoids $H$ with $\mathsf c(H)=\infty$ (\cite[Proposition 3.7]{Ge-Ha-Le07}), in contrast to the domain case as given in Theorem~\ref{5.1}.1.

Let $H$ be a finitely primary monoid of rank $s\in\mathbb N$ such that there exist some exponent $\alpha\in\mathbb N$ of $H$ and some system $\{p_i\mid i\in [1,s]\}$ of representatives of the prime elements of $\widehat H$ with the following property: for all $i\in [1,s]$ and for all $a\in\widehat H$ with $\mathsf v_{p_i}(a)\ge\alpha$ we have $p_ia\in H$ if and only if $a\in H$. Then $H$ is said to be
\begin{itemize}
\item {\it strongly ring-like} if ${\widehat H}^{\times}/H^{\times}$ is finite and $\{(\mathsf v_{p_i}(a))_{i=1}^s\mid a\in H\setminus H^{\times}\}\subset\mathbb N^s$ has a smallest element with respect to the partial order.
\end{itemize}
The concept of strongly ring-like monoids was introduced by Hassler (\cite{Ha09c}), and the question which one-dimensional local domains are strongly ring-like was studied in \cite[Section 5]{Ge-Re19d}.

A {\it numerical monoid} is a submonoid of $(\N_0,+)$ with finite complement, whence numerical monoids are finitely primary of rank one. Conversely, if $H\subset F=F^{\times}\times\mathcal F(\{p\})$ is finitely primary of rank one, then its {\it value monoid} $\mathsf v_p(H)=\{\mathsf v_p(a)\mid a\in H\}\subset\N_0$ is a numerical monoid.

\begin{proposition}\label{5.2}
Let $R$ be a local stable Mori domain with $(R\DP\widehat R)\ne\{0\}$. Then $R^{\bullet}$ is finitely primary of rank $s\le 2$ and it is strongly ring-like. If $s=2$, then $\rho(R^{\bullet})=\infty$ and if $s=1$ and $\mathfrak X(\widehat R)=\{\mathfrak p\}$, then the elasticity $\rho(R^{\bullet})$ is accepted with $\rho(R^{\bullet})=\max\mathsf v_{\mathfrak p}(R^{\bullet})/\min\mathsf v_{\mathfrak p}(R^{\bullet})$.
\end{proposition}

\begin{proof}
By Corollary~\ref{3.2}.1, $R$ is one-dimensional. By \cite[Proposition 2.10.7]{Ge-HK06a}, one-dimensional local Mori domains with non-zero conductor are finitely primary of rank $|\mathfrak X(\widehat R)|$. By Corollary~\ref{3.2}.2, $\overline R$ is a Dedekind domain with at most two maximal ideals, whence $s=|\mathfrak X(\widehat R)|\le2$. Since $(R\DP\overline R)\ne\{0\}$, every ideal of $R$ is $2$-generated by Proposition~\ref{3.5}.4, whence $R$ is noetherian.
If $\mathfrak m$ is the maximal ideal of $R$, then $\widehat R=\overline R$ and $|\max(\overline R)|\le 2\le |R/\mathfrak m|$, whence $R^{\bullet}$ is strongly ring-like by \cite[Corollary 5.7]{Ge-Re19d}.
If $s=2$, then $\rho(R^{\bullet})=\infty$ by \eqref{s=2}. Suppose that $s=1$. Since $R^{\bullet}$ is strongly ring-like, ${\widehat R}^{\times}/R^{\times}$ is finite and thus the elasticity is accepted and has the asserted value by \cite[Lemma 4.1]{Ge-Zh18a}.
\end{proof}

Let $R$ be a one-dimensional local Mori domain with $(R\DP\widehat R)\ne\{0\}$. If $R$ is stable, then, by Proposition~\ref{5.2}, we have $|\mathfrak X(\widehat R)|\le 2$. Example~\ref{5.5} shows that the converse does not hold in general. Example~\ref{5.4} and Proposition~\ref{5.7}.1 show that also for stable domains the exponent of $R^{\bullet}$ can be arbitrarily large. We start with a lemma.

\begin{lemma}\label{5.3}
Let $R$ be a Mori domain and a G-domain and let $I$ be a divisorial stable ideal of $R$. Then $I^2=xI$ for some $x\in I$.
\end{lemma}

\begin{proof}
Since every overring of a G-domain is a G-domain, $(I\DP I)$ is a G-domain. Since $I$ is divisorial and $R$ is a Mori domain, we have that $(I\DP I)$ is a Mori domain. Therefore, $\spec((I\DP I))$ is finite by \cite[Theorem 2.7.9]{Ge-HK06a}, and hence $(I\DP I)$ is semilocal. Consequently, $I=x(I\DP I)$ for some $x\in I$, and thus $I^2=xI$.
\end{proof}

\begin{example}[{\bf Stable orders in number fields}]\label{5.4}
1. Let $K=\mathbb{Q}(\sqrt{d})$ be a quadratic number field, where $d\in\mathbb{Z}\setminus\{0,1\}$ is squarefree, and let
\[
\omega=\begin{cases}\sqrt{d},&\text{if $d\equiv 2,3\mod 4$}\\\frac{1+\sqrt{d}}{2},&\text{if $d\equiv 1\mod 4$}\end{cases}\,.
\]
Let $R=\mathbb{Z}+p^n\omega\mathbb{Z}$, where $p\in\N$ is a prime number and $n\in\N$. Since every ideal of $R$ is $2$-generated, $R$ is a stable order in the Dedekind domain $\overline{R}=Z+\omega\mathbb{Z}$. Then $\mathfrak m=p\mathbb{Z}+p^n\omega\mathbb{Z}\in\mathfrak X(R)$ and $R_{\mathfrak m}$ is a one-dimensional local stable domain with non-zero conductor. By Corollary~\ref{3.2}, $R_{\mathfrak m}^{\bullet}$ is Mori, whence it is finitely primary of rank $s=|\{\mathfrak q\in\mathfrak X(\overline{R})\mid\mathfrak q\cap R=\mathfrak m\}|\leq 2$. Moreover, if $\alpha\in\N$ is the exponent of $R_{\mathfrak m}^{\bullet}$, then $\alpha\ge\max\{\mathsf v_{\mathfrak q}((R\DP\overline{R}))\mid\mathfrak q\in\mathfrak X(\overline{R}),\mathfrak q\cap R=\mathfrak m\}$ and since $(R\DP\overline{R})=p^n\overline{R}$, we obtain that $\alpha\ge n\max\{\mathsf v_{\mathfrak q}(p\overline{R})\mid\mathfrak q\in\mathfrak X(\overline{R}),\mathfrak q\cap R=\mathfrak m\}\geq n$.

\smallskip
2. Let $K$ be an algebraic number field, $\mathcal O_K$ its ring of integers, and $R\subset\mathcal O_K$ an order. If the discriminant $\Delta(R)\in\Z$ is not divisible by the fourth power of a prime, then $R$ is stable by a result of Greither (\cite[Theorem 3.6]{Gr82a}). In particular, if $a\in\N$ is squarefree with $3\nmid a$ and $R=\Z [\sqrt[3]{a}]\subset\Q(\sqrt[3]{a})$, then $\Delta(R)=27a^2$ is not divisible by a fourth power of a prime (for more on $R$ and $\mathcal O_K$ in the case of pure cubic fields we refer to \cite[Theorem 3.1.9]{HK20a}).
\end{example}

Next we discuss the catenary degree of finitely primary monoids, which has received a lot of attention in the literature. Let $H\subset F$ be a finitely primary monoid of rank $s$ and exponent $\alpha$, with all notation as in \eqref{finitely-primary}. Then the catenary degree is bounded above by $3\alpha-1$ in case $s=1$ and by $2\alpha+1$ otherwise. These bounds can be attained, but the catenary degree can also be much smaller. Indeed, as shown in Example~\ref{5.4}.2, for every $n\in\N$ there is an order $R$ in a quadratic number field whose localization $R_{\mathfrak p}$ at a maximal ideal $\mathfrak p$ is finitely primary of exponent greater than or equal to the given $n$ but the catenary degree $\mathsf c(R_{\mathfrak p})$ is bounded by $5$ (\cite[Theorem 1.1]{Br-Ge-Re20}). Let $H\subset F=F^{\times}\times\mathcal F(\{p\})$ be finitely primary of rank one, suppose that its value monoid $\mathsf v_p(H)=\{\mathsf v_p(a)\mid a\in H\}=\langle d_1,\ldots, d_s\rangle$, with $s\in\N$, $1<d_1<\ldots<d_s$, and $\gcd(d_1,\ldots, d_s)=1$. The catenary degree of numerical monoids has been studied a lot in recent literature (see \cite{GS-ON-We19, Om12a, ON-Pe17a, ON-Pe18a, Ph15a}, for a sample). By \eqref{basic-inequality-1}, we have $2+\max\Delta(H)\le\mathsf c(H)$. There are also results for $\min\Delta(H)$. Indeed,
by \cite[Lemma 4.1]{Ge-Zh18a}, we have
\[
\gcd(d_i-d_{i-1}\mid i\in [2,s])\mid\min\Delta(H)\quad\text{and if $|F^{\times}/H^{\times}|=1$, then}\quad\gcd(d_i-d_{i-1}\mid i\in [2,s])=\min\Delta(H)\,.
\]

We continue with examples of numerical semigroup rings and numerical power series rings. Let $K$ be a field and $H\subset\N_0$ be a numerical monoid. Then
\[
K[H]=K[X^h\mid h\in H]\subset K[X]\quad\text{and}\quad K\LK H\RK=K\LK X^h\mid h\in H\RK\subset K\LK X\RK
\]
denote the {\it numerical semigroup ring} and the {\it numerical power series ring}. Since $H$ is finitely generated, $K[H]$ is a one-dimensional noetherian domain with integral closure $K[X]$. The power series ring $K\LK H\RK$ is a one-dimensional local noetherian domain with integral closure $K\LK X\RK$, and its value monoid $\mathsf v_X(K\LK H\RK^{\bullet})$ is equal to $H$.

\begin{example}\label{5.5}

Let $K$ be a field and $H\subset\N_0$ be a numerical monoid distinct from $\N_0$. Then $H$ is not half-factorial, whence \eqref{basic-inequality-1} implies that $\mathsf c(H)\ge 3$. If $\min(H\setminus\{0\})\ge 3$, then $X^2\notin X K\LK H\RK+K\LK H\RK$, whence $R\subset\overline R$ is not a quadratic extension and $K\LK H\RK$ is not stable by Corollary~\ref{3.2}.2.

1. Let $H=\langle e, e+1,\ldots, 2e-1\rangle=\N_{\ge e}\cup\{0\}$ with $e\in\N_{\ge 2}$ and $R=K\LK H\RK$. By \cite[Special case 3.1, page 216]{Ge-HK06a}, we have $\mathsf c(H)=\mathsf c(R)=3$. Indeed, by \cite[Theorem 5.6]{Ph15a}, there is a transfer homomorphism $\theta\colon R^{\bullet}\to H$.

2. Let $K$ be finite, $H=\langle e, e+1,\ldots, 2e-1\rangle=\N_{\ge e}\cup\{0\}$ with $e\in\N_{\ge 2}$, and $R=K [H]$. We set $\rho=X+X^e\widehat R\in\widehat R/X^e\widehat R$ and $G=K [\rho]^{\times}/ K^{\times}$. Then, by \cite[Special Case 3.2, page 216]{Ge-HK06a}, we have $|G|=|K|^{e-1}$ and $\mathsf c(R)\ge\mathsf c(\mathcal B(G))$, where $\mathcal B(G)$ is the monoid of zero-sum sequences over $G$. Since $\mathsf c(\mathcal B(G))\ge\max\{\exp(G),1+\mathsf r(G)\}$ by \cite[Theorem 6.4.2]{Ge-HK06a}, the catenary degree of $R$ grows with $|G|$.
\end{example}

\begin{lemma}\label{5.6}
Let $R$ be an order in a Dedekind domain such that $R$ is a maximal proper subring of $\overline{R}$. Then we have
\begin{enumerate}
\item Every maximal ideal of $R$ is stable.
\item $R$ is stable if and only if $\overline{R}/R$ is a simple $R$-module.
\end{enumerate}
\end{lemma}

\begin{proof}
1. Let $\mathfrak m\in\mathfrak X(R)$. Then $(\mathfrak m\DP\mathfrak m)$ is an intermediate ring of $R$ and $\overline{R}$, and hence $(\mathfrak m\DP\mathfrak m)\in\{R,\overline{R}\}$. If $(\mathfrak m\DP\mathfrak m)=\overline{R}$, then $\mathfrak m$ is clearly an invertible ideal of $(\mathfrak m\DP\mathfrak m)$, since $\overline{R}$ is a Dedekind domain. Now let $(\mathfrak m\DP\mathfrak m)=R$. Since $\mathfrak m$ is divisorial, we have that $R=(\mathfrak m\DP\mathfrak m)=((R\DP \mathfrak m^{-1})\DP \mathfrak m)=(R\DP \mathfrak m\mathfrak m^{-1})=(\mathfrak m\mathfrak m^{-1})^{-1}$, and thus $\mathfrak m$ is $v$-invertible. Consequently, $\mathfrak m$ is invertible.

2. First let $R$ be stable. Then $R\subset\overline{R}$ is a quadratic extension by Proposition~\ref{3.3}.1. Let $N$ be an $R$-submodule of $\overline{R}$ with $R\subset N$. Then $N$ is an intermediate ring of $R$ and $\overline{R}$. Consequently, $N\in\{R,\overline{R}\}$, and hence $\overline{R}/R$ is a simple $R$-module.

Conversely, let $\overline{R}/R$ be a simple $R$-module. Obviously, $R\subset\overline{R}$ is a quadratic extension. Since $\overline{R}/R$ and $R/(R\DP\overline{R})$ are isomorphic as $R$-modules, we have that $(R\DP\overline{R})\in\mathfrak X(R)$. Let $\mathfrak m\in\mathfrak X(R)$. If $\mathfrak m\not=(R\DP\overline{R})$, then $R_{\mathfrak m}$ is a discrete valuation domain, and hence there is precisely one maximal ideal of $\overline{R}$ lying over $\mathfrak m$. Now let $\mathfrak m=(R\DP\overline{R})$ and set $k=|\{\mathfrak q\in\mathfrak X(\overline{R})\mid\mathfrak q\cap R=\mathfrak m\}|$. Assume that $k\geq 3$. Then there are some distinct $\mathfrak q_1,\mathfrak q_2,\mathfrak q_3\in\mathfrak X(\overline{R})$ such that $\mathfrak q_1\cap R=\mathfrak q_2\cap R=\mathfrak q_3\cap R$. Therefore, $\mathfrak m\subset\mathfrak q_1\cap\mathfrak q_2\cap\mathfrak q_3\subsetneq \mathfrak q_1\cap\mathfrak q_2\subsetneq\mathfrak q_1\subsetneq\overline{R}$, and thus $\boldsymbol l_R(\overline{R}/\mathfrak m)\geq 3$. On the other hand $\boldsymbol l_R(\overline{R}/\mathfrak m)=\boldsymbol l_R(\overline{R}/R)+\boldsymbol l_R(R/\mathfrak m)=2$, a contradiction. Consequently, $k\leq 2$ and thus $R$ is finitely stable by Proposition~\ref{3.3}.1. Since $R$ is noetherian, we have that $R$ is stable.
\end{proof}

In the next proposition we study the catenary degree of finitely primary monoids stemming from one-dimensional local stable domains. We establish an upper bound for their catenary degree in case when $R\subset\overline R$ is a maximal proper subring.

(i) Let $H$ be a finitely primary monoid of rank one. In general, the map $\theta\colon H\to\mathsf v_p(H)$, $a\mapsto\mathsf v_p(a)$, need not be a transfer homomorphism.
(Example: If $H=[\varepsilon_1 p,\varepsilon_2 p, p^2 ]\subset F^{\times}\times\mathcal F(\{p\})$ with $\varepsilon_i\varepsilon_j\ne 1$ for all $i,j\in [1,2]$).

(ii) Let us consider the following example: let $H$ be a reduced finitely primary monoid of rank one, say
\[
H\subset F^{\times}\times\mathcal F(\{p\})\,.
\]
Suppose that $H$ is generated by the following $k+1$ elements, where $k$ is even:
\[
\varepsilon_1 p ,\ldots,\varepsilon_k p, p^2\,,
\]
where $\varepsilon_1\cdot\ldots\cdot\varepsilon_k$ is a minimal product-one sequence in the group $F^{\times}$.

Then
\[
(\varepsilon_1 p)\cdot\ldots\cdot (\varepsilon_k p)=p^2\cdot\ldots\cdot p^2\quad\text{($k/2$ times)}
\]
is a minimal relation of atoms of $H$, whence $\mathsf c(H)\ge k$.

\smallskip
\begin{proposition}\label{5.7}
Let $R$ be a local stable order in a Dedekind domain. Define $R_0=R$.
\begin{enumerate}
\item Suppose that $R_0\subsetneq R_1\subsetneq\dots\subsetneq R_n=\overline{R}$ where $R_i=(\mathfrak m_{i-1}\DP\mathfrak m_{i-1})$ and $R_{i-1}$ is local with maximal ideal ${\mathfrak m}_{i-1}$ for all $i\in [1,n]$, and $\mathfrak X(R_n)=\{{\mathfrak P}_1,{\mathfrak P}_2\}$. Write ${\mathfrak m}_0=R_1m_0$ for some $m_0\in {\mathfrak m}_0\setminus\{0\}$. Then $(R\DP\overline{R})=m_0^n\overline{R}={\mathfrak P}_1^n {\mathfrak P}_2^n$, and $R^{\bullet}$ is a finitely primary monoid of rank at most two and exponent $n$.
\item If $R\subset\overline{R}$ is a maximal proper subring, then $\mathsf c(R)\le 5$.
\end{enumerate}
\end{proposition}

\begin{proof}
1. All domains $R_0,\ldots, R_{n-1}$ are local with maximal ideals ${\mathfrak m}_i$ such that ${\mathfrak m}_i={\mathfrak m}_0R_{i+1}=m_0R_{i+1}$ by \cite[Proposition 4.2]{Ol02a} and also the Jacobson radical of $R_n$, $J_n={\mathfrak P}_1\cap {\mathfrak P}_2={\mathfrak P}_1{\mathfrak P}_2={\mathfrak m}_{n-1}$ and for some $k>0$, $J_n^k=({\mathfrak P}_1{\mathfrak P}_2)^k={\mathfrak m}_{n-1}^k\subset {\mathfrak m}_0$ by \cite[Corollary 4.4]{Ol02a}. Therefore, ${\mathfrak m}_{n-1}^{k}R_n\subset {\mathfrak m}_0\subset R$, i.e., $(R\DP R_n)={\mathfrak m}_{n-1}^k={\mathfrak P}_1^k{\mathfrak P}_2^k$ and since ${\mathfrak m}_{n-1}=m_0R_n$, $(R\DP R_n)=m_0^kR_n$. Also $m_0^nR_n=m_0^{n-1}m_0R_n=m_0^{n-1}{\mathfrak m}_{n-1}\subset m_0^{n-1}R_{n-1}\subset\dots\subset m_0R_1={\mathfrak m}\subset R$ and $(R\DP R_n)=m_0^nR_n$. Therefore, $(R\DP\overline{R})=m_0^n\overline{R}={\mathfrak P}_1^n{\mathfrak P}_2^n$ and $R^{\bullet}$ is a finitely primary monoid of rank two and exponent $n$.

2. Let $\mathfrak m$ denote the maximal ideal of $R$. Since $|\max(\overline{R})|\le 2$ by Proposition~\ref{3.3}.1, we distinguish two cases.
	
First, suppose that $\overline{R}$ is local with maximal ideal $\mathfrak P$. Then by \cite[Proposition 4.2 (i)]{Ol02a}, $\mathfrak P^2 \subset\mathfrak m$, which implies that $\mathfrak P^2\overline{R}\subset\mathfrak m\subset R$ and hence $(R\DP\overline{R})=\mathfrak P^k$ with $k\in\{1,2\}$. Thus $R^\bullet$ is finitely primary of rank one and exponent two, whence $\mathsf c(R)\le 5$ by \eqref{s=1}.
	
Second, suppose that $\max(\overline{R})=\{\mathfrak P_1,\mathfrak P_2\}$. Then 1. shows that $(R\DP\overline{R})=\mathfrak P_1\mathfrak P_2$. Thus $R^\bullet$ is finitely primary of rank two and exponent one, whence $\mathsf c(R)\le 3$ by \eqref{s=2}.
\end{proof}

For an atomic monoid $H$, we set
\[
\daleth(H)=\sup\{\min(L\setminus\{2\})\mid 2\in L\in\mathcal L(H),|L|>1\}\,.
\]
Then $\daleth(H)=0$ if and only if $\mathsf L(uv)=\{2\}$ for all $u,v\in\mathcal A(H)$, and $\daleth(H)\ge 3$ otherwise. If $H$ is not half-factorial, then
\begin{equation}\label{basic-inequality-2}
\daleth(H)\le 2+\sup\Delta(H)\,.
\end{equation}
The question of whether equality holds was studied a lot. Among others, equality holds for large classes of Krull domains (\cite[Corollary 4.5]{Ge-Gr-Sc11a}), for numerical monoids $H$ with $|\mathcal A(H)|=2$, but not for all finitely primary monoids.

\begin{proposition}\label{5.8}
Let $R$ be a local domain with maximal ideal $\mathfrak m$ such that $R$ is not a field and $\bigcap_{n\in\N_0} \mathfrak m^n=\{0\}$, let $x\in\mathfrak m$ be such that $\mathfrak m^2=x\mathfrak m$ and let $U=xR$.
\begin{enumerate}
\item $\mathcal{I}(R)$ is a reduced atomic monoid, $U$ is a cancellative atom of $\mathcal{I}(R)$ and for every $I\in\mathcal{I}(R)\setminus\{R\}$ there are $n\in\N_0$ and $J\in\mathcal{A}(\mathcal{I}(R))$ such that $I=U^nJ$.
\item $\daleth(\mathcal{I}(R) )=2+\sup\Delta(\mathcal{I}(R) )$ and $\daleth(\mathcal{I^*}(R) )=2+\sup\Delta(\mathcal{I^*}(R) )$.
\end{enumerate}
\end{proposition}

\begin{proof}
1. It follows from Lemma~\ref{4.1} that $\mathcal{I}(R)$ is an atomic monoid. Since $R$ is not a field, we have that $U$ is a non-zero proper ideal of $R$. If $I$ and $J$ are non-zero ideals of $R$ such that $UI=UJ$, then $xI=xJ$, and hence $I=J$. Therefore, $U$ is cancellative. Assume that $U$ is not an atom of $\mathcal{I}(R)$. Then there are some proper $A,B\in\mathcal{I}(R)$ such that $U=AB$. We infer that $xR\subset\mathfrak m^2=x\mathfrak m$. Consequently, $x=xu$ for some $u\in\mathfrak m$, and thus $1=u\in\mathfrak m$, a contradiction. This implies that $U$ is an atom of $\mathcal{I}(R)$. Now let $I$ be a non-zero proper ideal of $R$. Then $I\subset\mathfrak m$, and since $\bigcap_{n\in\N_0}\mathfrak m^n=\{0\}$, there is some $m\in\N$ such that $I\subset\mathfrak m^m$ and $I\not\subset\mathfrak m^{m+1}$. We infer that $I\subset x^{m-1}\mathfrak m$. Set $n=m-1$. Then $n\in\N_0$ and there is some proper $J\in\mathcal{I}(R)$ such that $I=x^nJ=U^nJ$. Assume that $J$ is not an atom of $\mathcal{I}(R)$. Then there are some non-zero proper ideals $A$ and $B$ of $R$ with $J=AB$, and thus $J\subset\mathfrak m^2=x\mathfrak m$. Therefore, $I=x^nJ\subset x^m\mathfrak m=\mathfrak m^{m+1}$, a contradiction. It follows that $J\in\mathcal{A}(\mathcal{I}(R))$.

2. This follows from 1. and from \cite[Proposition 4.1]{Br-Ge-Re20}.
\end{proof}

\begin{theorem}\label{5.9}
Let $R$ be a one-dimensional Mori domain such that for every $\mathfrak m\in\mathfrak X(R)$, $\mathfrak m$ is stable and $(R_{\mathfrak m}\DP\widehat{R_{\mathfrak m}})\not=\{0\}$.
\begin{enumerate}
\item The following statements are equivalent.
\begin{enumerate}
\item[(a)] $\mathcal{I}(R)$ is transfer Krull.
\item[(b)] $\mathcal{I}^*(R)$ is transfer Krull.
\item[(c)] $\mathcal{I}^*(R)$ is half-factorial.
\item[(d)] $\mathcal{I}(R)$ is half-factorial.
\item[(e)] $\mathsf c(\mathcal{I}(R))\leq 2$.
\item[(f)] $\mathsf c(\mathcal{I}^*(R))\leq 2$.
\end{enumerate}
If these conditions hold, then the map $\pi\colon\mathfrak X(\widehat R)\to\mathfrak X(R)$, defined by $\mathfrak P\mapsto\mathfrak P\cap R$, is bijective.
\smallskip
\item $\daleth(\mathcal{I}(R))=2+\sup\Delta(\mathcal{I}(R))$ and $\daleth(\mathcal{I^*}(R))=2+\sup\Delta(\mathcal{I^*}(R))$.
\end{enumerate}
\end{theorem}

\begin{proof}
1. Suppose that Condition (c) holds. By Proposition~\ref{4.3}.2, $\mathcal I^*(R)$ is half-factorial if and only if $R_{\mathfrak p}$ is half-factorial for every $\mathfrak p\in\mathfrak X(R)$. Thus the map $\pi$ is bijective by Theorem~\ref{5.1}.1.

(a) $\Rightarrow$ (b) $\mathcal{I}^*(R)\subset\mathcal{I}(R)$ is a divisor-closed submonoid, and divisor-closed submonoids of transfer Krull monoids are transfer Krull.

(b) $\Rightarrow$ (c) Since $R$ is a one-dimensional Mori domain, we have that $\mathcal I_v^*(R)=\mathcal I^*(R)$, and thus the assertion follows from Theorem~\ref{5.1}.4.

(c) $\Rightarrow$ (d) Since $R$ is a one-dimensional Mori domain, we have that $R$ is of finite character. Furthermore, if $\mathfrak m\in\mathfrak X(R)$, then $R_{\mathfrak m}$ is a one-dimensional local Mori domain with non-zero conductor, and hence $R_{\mathfrak m}$ is finitely primary. By Corollary~\ref{4.5} it remains to show that for every $\mathfrak m\in\mathfrak X(R)$, $\mathfrak m^2$ is contained in a proper invertible ideal of $R$. Let $\mathfrak m\in\mathfrak X(R)$. Since $R$ is a Mori domain and $\mathfrak m(\mathfrak m\DP\mathfrak m^2)=(\mathfrak m\DP\mathfrak m)$, we infer that $\mathfrak m_{\mathfrak m}(\mathfrak m_{\mathfrak m}\DP\mathfrak m_{\mathfrak m}^2)=(\mathfrak m_{\mathfrak m}\DP\mathfrak m_{\mathfrak m})$, i.e., $\mathfrak m_{\mathfrak m}$ is a stable ideal of $R_{\mathfrak m}$. Clearly, $\mathfrak m_{\mathfrak m}$ is a divisorial ideal of $R_{\mathfrak m}$. It follows from Lemma~\ref{5.3} that $\mathfrak m_{\mathfrak m}^2=x\mathfrak m_{\mathfrak m}$ for some $x\in \mathfrak m_{\mathfrak m}$. Observe that $\mathfrak m^2=\mathfrak m_{\mathfrak m}^2\cap R\subset xR_{\mathfrak m}\cap R$. Moreover, $xR_{\mathfrak m}\cap R$ is $t$-finitely generated and locally principal and $xR_{\mathfrak m}\cap R\subset\mathfrak m$, and thus $xR_{\mathfrak m}\cap R$ is a proper invertible ideal of $R$.

(d) $\Rightarrow$ (a) All half-factorial monoids are transfer Krull.

(d) $\Rightarrow$ (e) Let $\mathfrak m\in\mathfrak X(R)$. By Proposition~\ref{4.3}.2 we have that $\mathcal{I}(R_{\mathfrak m})$ is half-factorial. Note that $R_{\mathfrak m}$ is a Mori domain and a G-domain. Since $\mathfrak m$ is stable and $R$ is a Mori domain, we have that $\mathfrak m_{\mathfrak m}$ is a stable ideal of $R_{\mathfrak m}$. Clearly, $\mathfrak m_{\mathfrak m}$ is a divisorial ideal of $R_{\mathfrak m}$. Therefore, $\mathfrak m_{\mathfrak m}^2=x\mathfrak m_{\mathfrak m}$ for some $x\in \mathfrak m_{\mathfrak m}$ by Lemma~\ref{5.3}. Since $R_{\mathfrak m}$ is a one-dimensional local Mori domain, it follows that $\bigcap_{n\in\N_0} \mathfrak m_{\mathfrak m}^n=\{0\}$. We infer by Proposition~\ref{5.8} and \cite[Proposition 4.1.4]{Br-Ge-Re20} that $\mathsf c(\mathcal{I}(R_{\mathfrak m}))\leq 2$. Therefore, $\mathsf c(\mathcal{I}(R))\leq 2$ by Theorem~\ref{5.1}.

(e) $\Rightarrow$ (f) This is obvious, since $\mathcal{I}^*(R)$ is a divisor-closed submonoid of $\mathcal{I}(R)$.

(f) $\Rightarrow$ (c) Since $\mathcal{I}^*(R)$ is cancellative, this follows from \eqref{basic-inequality-1}.

2. If $(H_i)_{i\in I}$ is a family of atomic monoids, then
\[
\sup\Delta\big(\coprod_{i\in I} H_i\big)=\sup\{\sup\Delta(H_i)\mid i\in I\}\quad\text{and}\quad\daleth\big(\coprod_{i\in I} H_i\big)=\sup\{\daleth(H_i)\mid i\in I\}\,.
\]
Thus the claim follows from Propositions~\ref{4.3} and~\ref{5.8}.2.
\end{proof}

Let $R$ be as in Theorem~\ref{5.9}. Clearly, we have $\daleth(\mathcal{I^*}(R))\le\daleth(\mathcal{I}(R))$, but in general we do not have equality.

By Theorem~\ref{3.7}, stable domains with non-zero conductor, that are Mori or weakly Krull, are already orders in Dedekind domains. Thus our next result is formulated for stable orders in Dedekind domains.
Its first part generalizes a result valid for orders in quadratic number fields (\cite[Theorem 1.1]{Br-Ge-Re20}). Note, if $R$ is a semilocal domain, then $\Pic(R)=\boldsymbol 0$. This means that every invertible ideal is principal, whence $\mathcal I^*(R)=\{ aR\mid a\in R^{\bullet}\}\cong(R^{\bullet})_{\red}$. If $R$ is not semilocal, then the statements for $\mathcal I^*(R)$ need not hold for $R$. If $R$ is any order in an algebraic number field, then $\mathsf c(R)\ge\mathsf c\big(\mathcal B(\Pic(R))\big)$ (\cite[Sections 3.4 and 3.7]{Ge-HK06a}). Moreover, $R$ can be transfer Krull without being half-factorial (\cite[Theorems 5.8 and 6.2]{Ge-Ka-Re15a}).

\smallskip
\begin{theorem}\label{5.10}
Let $R$ be a stable order in a Dedekind domain.
\begin{enumerate}
\item The following statements are equivalent.
\begin{enumerate}
\item[(a)] $\mathcal{I}(R)$ is transfer Krull.
\item[(b)] $\mathcal{I}^*(R)$ is transfer Krull.
\item[(c)] $\mathcal{I}^*(R)$ is half-factorial.
\item[(d)] $\mathcal{I}(R)$ is half-factorial.
\item[(e)] $\mathsf c(\mathcal{I}(R))\leq 2$.
\item[(f)] $\mathsf c(\mathcal{I}^*(R))\leq 2$.
\end{enumerate}
\item $\mathsf c_{\mon}\big(\mathcal I^*(R)\big)<\infty$.
\item $\mathcal I^*(R)$ has finite elasticity if and only if $\pi\colon\mathfrak X(\overline R)\to\mathfrak X(R)$ is bijective. If this holds, then the elasticity is accepted.
\end{enumerate}
\end{theorem}

\begin{proof}
1. This is an immediate consequence of Theorem~\ref{5.9}.

2. Since $R$ is an order in a Dedekind domain, $R$ is a weakly Krull Mori domain with non-zero conductor. By Proposition~\ref{5.2}, the localizations $R_{\mathfrak p}$ are strongly ring-like of rank at most two. Thus $\mathcal I^*(R)$ has finite monotone catenary degree by \cite[Theorem 5.13]{Ge-Re19d}.

3. By \cite[Proposition 1.4.5]{Ge-HK06a} and by Proposition~\ref{4.3}, we have
\[
\rho\big(\mathcal I^*(R)\big)=\sup\{\rho(R_{\mathfrak p})\mid\mathfrak p\in\mathfrak X(R)\}\,.
\]
Thus the assertion follows from Proposition~\ref{5.2}.
\end{proof}

In Remark~\ref{5.11} we briefly discuss further arithmetical properties, which follow from the ones given in Theorem~\ref{5.10}. Then we work out, in a series of remarks, that
none of the statements in Theorem~\ref{5.10} holds true in general without the stability assumption.

\begin{remark}[{\bf Structure of sets of lengths and of their unions}]\label{5.11}

1. (Structure of sets of lengths) If $R$ is an order in a Dedekind domain, then sets of lengths of $\mathcal I^*(R)$ are well-structured. They are almost arithmetical multiprogressions with global bounds for all parameters (\cite[Section 4.7]{Ge-HK06a}). This holds without the stability assumption.

2. (Structure of unions of sets of lengths) Let $H$ be an atomic monoid. For every $k\in\N$,
\[
\mathcal U_k(H)=\bigcup_{k\in L\in\mathcal L(H)} L\quad\subset\N
\]
is the {\it union of sets of lengths} containing $k$. The structure theorem for unions of sets of lengths states that there is a bound $M$ such that almost all sets $\mathcal U_k(H)\cap [\min\mathcal U_k(H)+M,\max\mathcal U_k(H)-M]$ are arithmetical progressions with difference $\min\Delta(H)$. Now
every atomic monoid with accepted elasticity satisfies this structure theorem for unions of sets of lengths, and the initial parts $\mathcal U_k(H)\cap [\min\mathcal U_k(H),\min\mathcal U_k(H)+M]$ and the end parts $[\max\mathcal U_k(H)-M,\max\mathcal U_k(H)]$ fulfill a periodicity property (we refer to recent work of Tringali (\cite[Theorem 1.2]{Tr19a}).
\end{remark}

\begin{remark}[{\bf On catenary degrees}]\label{5.12}
Example~\ref{5.5}.2 offers examples of non-stable orders in Dedekind domains whose catenary degree is arbitrarily large. Furthermore, there are finitely primary monoids with arbitrarily large catenary degree (see the discussion after Lemma \ref{5.6}). Non-stable local orders in Dedekind domains may have infinite monotone catenary degree (\cite[Examples 6.3 and 6.5]{Ha09c}).
\end{remark}

\begin{remark}[{\bf Seminormal orders}]\label{5.13}
We compare the arithmetic of stable orders with the arithmetic of seminormal orders in Dedekind domains. Note that stable orders need not be seminormal (all orders in quadratic number fields are stable but not all are seminormal \cite{Do-Fo87}) and seminormal orders need not be stable (see the example given in Remark~\ref{5.15}).

Let $R$ be a seminormal order in a Dedekind domain and let $\pi\colon\mathfrak X(\overline R)\to\mathfrak X(R)$ be defined by $\pi(\mathfrak P)=\mathfrak P\cap R$ for all $\mathfrak P\in\mathfrak X(\overline R)$. If $\pi$ is bijective, then $\mathsf c\big(\mathcal I^*(R)\big)=2$. If $\pi$ is not bijective, then $\mathsf c\big(\mathcal I^*(R)\big)=3$ and $\mathsf c_{\mon}\big(\mathcal I^*(R)\big)\in\{3,5\}$ (\cite[Theorem 5.8]{Ge-Ka-Re15a}). Furthermore, $\mathcal I^*(R)$ is half-factorial if and only if $\pi$ is bijective. For stable orders, only one implication is true (see Theorem~\ref{5.9}).
\end{remark}

\begin{remark}[{\bf Half-factoriality of $\mathcal I^*(R)$ does not imply half-factoriality of $\mathcal I(R)$}]\label{5.14}~

The Statements 1.(c) and 1.(d) of Theorem~\ref{5.10} need not be equivalent for divisorial orders in Dedekind domains. We construct a local divisorial order $R$ in a Dedekind domain such that $\mathcal{I}^*(R)$ is half-factorial, and yet $\mathcal{I}(R)$ is not half-factorial.

Let $L$ be the field with $16$ elements, let $K\subset L$ be the field with $2$ elements, let $y\in L$ be such that $y^4=1+y$ and let $V=(1,y,y^2)_K$. Let $X$ be an indeterminate over $L$ and let $R=K+VX+X^2L\LK X\RK$. We assert that $R$ is a local divisorial half-factorial Cohen-Kaplansky domain such that $\mathcal{I}(R)$ is not half-factorial.

\begin{proof}
By \cite[Example 6.7]{An-Mo92} we have that $R$ is a half-factorial local Cohen-Kaplansky domain, $(1,y,y^2,y^3)$ is a $K$-basis of $L$ and $L=\{ab\mid a,b\in V\}$. Let $\mathfrak m=VX+X^2L\LK X\RK$ and let $I=\langle yX^2,(1+y^3)X^2\rangle_R$. Then $\mathfrak m$ is the maximal ideal of $R$. Note that $\mathfrak m^{-1}=(\mathfrak m\DP\mathfrak m)=\{f\in L\LK X\RK\mid f(V+XL\LK X\RK)\subset V+XL\LK X\RK\}=\{f\in L\LK X\RK\mid f_0V\subset V\}=K+XL\LK X\RK=\langle 1,y^3X\rangle_R$, and hence $R$ is divisorial by \cite[Theorem 3.8]{Ma68a}. By Proposition~\ref{4.2}.2 it is sufficient to show that $I\in\mathcal{A}(\mathcal{I}(R))$ and $I\subset\mathfrak m^2$. Observe that $\mathfrak m^2=X^2L\LK X\RK$, $\mathfrak m^3=X^3L\LK X\RK$ and $I=\{0,y,1+y^3,1+y+y^3\}X^2+X^3L\LK X\RK$. Therefore, $\mathfrak m^3\subset I\subset\mathfrak m^2$. For every ideal $E$ of $R$ let $S(E)=\{a\in V\mid aX+z\in E$ for some $z\in\mathfrak m^2\}$ and $T(E)=\{a\in L\mid aX^2+z\in E$ for some $z\in\mathfrak m^3\}$. Note that $S(E)$ is a $K$-subspace of $V$ and $T(E)$ is a $K$-subspace of $L$.

\smallskip
Claim: If $A$ and $B$ are proper ideals of $R$, then $T(AB)=(S(A)S(B))_K$.

Let $A$ and $B$ be proper ideals of $R$. First let $a\in T(AB)$. Then $aX^2+z\in AB$ for some $z\in\mathfrak m^3$. Therefore, $aX^2+z=\sum_{i=1}^n f_ig_i$ for some $n\in\N$, $f_i\in A$ and $g_i\in B$ for every $i\in [1,n]$. Since $A,B\subset\mathfrak m$, there are some $a_i,b_i\in V$ and $z_i,v_i\in\mathfrak m^2$ for every $i\in [1,n]$ such that $f_i=a_iX+z_i$ and $g_i=b_iX+v_i$ for every $i\in [1,n]$. Consequently, $a_i\in S(A)$ and $b_i\in S(B)$ for all $i\in [1,n]$. Moreover, $aX^2+z=(\sum_{i=1}^n a_ib_i)X^2+\sum_{i=1}^n (a_iv_iX+b_iz_iX+z_iv_i)$. Since $\sum_{i=1}^n (a_iv_iX+b_iz_iX+z_iv_i)\in\mathfrak m^3$ this implies that $a=\sum_{i=1}^n a_ib_i\in (S(A)S(B))_K$.

Now let $a\in (S(A)S(B))_K$. Then $a=\sum_{i=1}^n a_ib_i$ with $n\in\N$ and $a_i\in S(A)$ and $b_i\in S(B)$ for every $i\in [1,n]$. There are some $z_i,v_i\in\mathfrak m^2$ for every $i\in [1,n]$ such that $a_iX+z_i\in A$ and $b_iX+v_i\in B$ for every $i\in [1,n]$. Therefore, $aX^2+\sum_{i=1}^n (a_iv_iX+b_iz_iX+z_iv_i)=\sum_{i=1}^n (a_iX+z_i)(b_iX+v_i)\in AB$. Since $\sum_{i=1}^n (a_iv_iX+b_iz_iX+z_iv_i)\in\mathfrak m^3$, we have that $a\in T(AB)$. This proves the claim.

Assume that $I\not\in\mathcal{A}(\mathcal{I}(R))$. Then there are proper ideals $A$ and $B$ of $R$ such that $I=AB$. It follows by the claim that $\{0,y,1+y^3,1+y+y^3\}=T(I)=(S(A)S(B))_K$. Clearly, $\dim_K(S(A)),\dim_K(S(B))>0$. If $\dim_K(S(A))=\dim_K(S(B))=1$, then $|(S(A)S(B))_K|=2$, a contradiction. Therefore, $\dim_K(S(A))\geq 2$ or $\dim_K(S(B))\geq 2$. Without restriction let $\dim_K(S(A))\geq 2$. There are some non-zero $a\in S(B)$ and some two-dimensional $K$-subspace $W$ of $S(A)$. We infer that $(S(A)S(B))_K\supset aW$ and $4=|(S(A)S(B))_K|\geq |aW|=|W|=4$, and thus $\{0,y,1+y^3,1+y+y^3\}=aW$. Clearly, $a\in\{1,y,1+y,y^2,1+y^2,y+y^2,1+y+y^2\}$. To obtain a contradiction it is sufficient to show that $W\not\subset V$.

\smallskip
Case 1: $a=1$. Then $W=\{0,y,1+y^3,1+y+y^3\}\not\subset V$.

\smallskip
Case 2: $a=y$. Then $W=(1+y^3)\{0,y,1+y^3,1+y+y^3\}=\{0,1,1+y^2+y^3,y^2+y^3\}\not\subset V$.

\smallskip
Case 3: $a=1+y$. Then $W=(y+y^2+y^3)\{0,y,1+y^3,1+y+y^3\}=\{0,1+y+y^2+y^3,1+y+y^2,y^3\}\not\subset V$.

\smallskip
Case 4: $a=y^2$. Then $W=(1+y^2+y^3)\{0,y,1+y^3,1+y+y^3\}=\{0,1+y^3,1+y+y^2+y^3,y+y^2\}\not\subset V$.

\smallskip
Case 5: $a=1+y^2$. Then $W=(1+y+y^3)\{0,y,1+y^3,1+y+y^3\}=\{0,1+y^2,y^2+y^3,1+y^3\}\not\subset V$.

\smallskip
Case 6: $a=y+y^2$. Then $W=(1+y+y^2)\{0,y,1+y^3,1+y+y^3\}=\{0,y+y^2+y^3,y+y^3,y^2\}\not\subset V$.

\smallskip
Case 7: $a=1+y+y^2$. Then $W=(y+y^2)\{0,y,1+y^3,1+y+y^3\}=\{0,y^2+y^3,1+y,1+y+y^2+y^3\}\not\subset V$.
\end{proof}
\end{remark}

\smallskip
\begin{remark}[{\bf Transfer Krull does not imply stability}]\label{5.15}
If $R$ is an order in a Dedekind domain with $(R\DP\overline R)\in\max(R)$ and $\overline R=R {\overline R}^{\times}$, then $R^{\bullet}$ is transfer Krull by \cite[Proposition 3.7.5]{Ge-HK06a}. We provide an example showing that such an order need not be stable.

We construct a seminormal one-dimensional local noetherian domain $R$ such that $\overline{R}=R\overline{R}^{\times}$, $(R\DP\overline{R})\in\max(\overline{R})$, $\overline{R}$ is local, and $R$ has ideals which are not $2$-generated. Thus Corollary~\ref{3.8} implies that $R$ is not stable.

Let $K\subset L$ be a field extension with $3\leq [L\DP K] <\infty$, $X$ be an indeterminate over $L$, and $R=K+XL\LK X\RK$. Observe that $\widehat{R}=L\LK X\RK$ is a completely integrally closed one-dimensional noetherian domain. Let $B\subset L$ be a $K$-basis of $L$. Then $\widehat{R}=\langle B\rangle_R$, and hence $\widehat{R}$ is a finitely generated $R$-module. Since $\widehat{R}$ is noetherian, it follows from the Theorem of Eakin-Nagata that $R$ is noetherian, and hence $\overline{R}=\widehat{R}$.
	
Since $\overline{R}$ is one-dimensional local and $R\subset\overline{R}$ is an integral extension, we have that $R$ itself is local and one-dimensional. Moreover, $\spec(R)=\{\{0\},XL\LK X\RK\}$ and $R$ is transfer Krull by \cite[Theorem 5.8]{Ge-Ka-Re15a}.
	
Now if $x\in\mathsf q(R)$ with $x^2,x^3\in R$, then $x^2,x^3\in\widehat{R}$, and hence $x\in\widehat{R}$ (since $\widehat{R}$ is completely integrally closed), so $x_{0}\in L$ and $x_{0}^2,x_{0}^3\in K$ (whence $x_{0}$ is the constant term of $x$), and thus if $x_{0}=0$, then $x\in R$ and if $x_{0}\neq 0$, then $x_{0}=x_{0}^3x_{0}^{-2}\in K$, hence $x\in R$. Therefore, $R$ is seminormal.
	
Note that $(R\DP\overline{R})=(R\DP\widehat{R})=XL\LK X\RK\neq\{0\}$ and $(R\DP\overline{R})\in\max(\overline{R})$. It is clear that $R\overline{R}^{\times}\subset\overline{R}$. Let $y\in\overline{R}=L\LK X\RK$. We have that $y-y_0\in XL\LK X\RK$ (where $y_0$ is the constant term of $y$). If $y\in R$, then clearly $y\in R\overline{R}^{\times}$. Now let $y\not\in R$. Then $y_0\not\in K$, and hence $y_0\not=0$. Observe that $(y-y_0)y_0^{-1}\in XL\LK X\RK$ and $y_0\in L^{\times}\subset\overline{R}^{\times}$. Therefore, $y=(1+(y-y_0)y_0^{-1})y_0\in R\overline{R}^{\times}$.
	
Assume to the contrary, that $XL\LK X\RK$ is $2$-generated. Therefore, there exist $x,y\in R$ with $XL\LK X\RK=\langle x,y\rangle_R$. Let $x_{1},y_{1}$ be the linear coefficients of $x$ respectively $y$. Then $L=\langle x_{1},y_{1}\rangle_K$ and hence $[L\DP K]\leq 2$, a contradiction.
\end{remark}

\medskip
\noindent
{\bf Acknowledgements.} We would like to thank the anonymous referee for many valuable suggestions and comments which improved the quality of this paper.

\providecommand{\bysame}{\leavevmode\hbox to3em{\hrulefill}\thinspace}
\providecommand{\MR}{\relax\ifhmode\unskip\space\fi MR}
\providecommand{\MRhref}[2]{\href{http://www.ams.org/mathscinet-getitem?mr=#1}{#2}}
\providecommand{\href}[2]{#2}

\end{document}